\documentclass[11pt]{article}

\usepackage[margin=1in]{geometry}

\usepackage{amsmath,amsthm,amsfonts,amssymb}
\usepackage{mathtools}
\usepackage{mathrsfs}

\usepackage{tikz}
\usetikzlibrary{intersections,calc}
\usepackage{hyperref}
\usepackage{xcolor}

\usepackage{comment}

\usepackage[english]{babel}
\usepackage[style=alphabetic]{biblatex}
\theoremstyle{plain}

\newtheorem{thm}{Theorem}
\newtheorem{cor}[thm]{Corollary}
\newtheorem{prop}[thm]{{\bf Proposition}}
\newtheorem{lem}[thm]{{\bf Lemma}}
\newtheorem{obs}[thm]{{\bf Observation}}

\newcounter{hyp_counter}
\theoremstyle{definition}
\newtheorem{claim}{Claim}
\newtheorem{defn}{Definition}[section]

\theoremstyle{remark}
\newtheorem{rem}{Remark}
\newtheorem{example}{Example}

\newcommand{\Aut}{\operatorname{Aut}}

\newcommand{\Exp}{\operatorname{Exp}}
\newcommand{\Id}{\operatorname{Id}}

\newcommand{\Diff}{\operatorname{Diff}}

\newcommand{\Q}{\mathbb{Q}}
\newcommand{\R}{\mathbb{R}}

\newcommand{\SL}{\operatorname{SL}}

\newcommand{\spann}{\operatorname{span}}

\newcommand{\vol}{\operatorname{vol}}

\newcommand{\Z}{\mathbb{Z}}

\newcommand{\wt}[1]{\widetilde{#1}}

\newcommand{\abs}[1]{\left| #1\right|}

\newcommand{\mc}[1]{\mathcal{#1}}
\newcommand{\mf}[1]{\mathfrak{#1}}

\newcommand{\pez}[1]{\left( #1\right)}

\newcommand{\lcc}[2]{#1_{#2}}                                %#2th term in lower central series of #1
\newcommand{\cgrad}[2]{#1_{(#2)}}                        %#2th term in the Carnot grading of #1

\makeatletter
\def\blfootnote{\xdef\@thefnmark{}\@footnotetext}
\makeatother

\title{Local Lyapunov Spectrum Rigidity of Nilmanifold Automorphisms}
\author{Jonathan DeWitt}

\date{\today}

\begin{document}

\maketitle

\begin{abstract}
We study the regularity of a conjugacy between an Anosov automorphism $L$ of a nilmanifold $N/\Gamma$ and a volume-preserving, $C^1$-small perturbation $f$. We say that $L$ is \emph{locally Lyapunov spectrum rigid} if this conjugacy is $C^{1+}$ whenever $f$ is $C^{1+}$ and has the same volume Lyapunov spectrum as $L$. For $L$ with simple spectrum, we show that local Lyapunov spectrum rigidity is equivalent to $L$ satisfying both an irreducibility condition and an ordering condition on its Lyapunov exponents.
\end{abstract}

\section{Introduction}

Since \blfootnote{This material is based upon work supported by the National Science Foundation Graduate Research Fellowship under Grant No. DGE-1746045.}their introduction, Anosov diffeomorphisms have been a central class of examples in the field of dynamical systems.
A diffeomorphism $f$ of a closed Riemannian manifold $M$ is \emph{Anosov} if the tangent bundle of $M$ splits into the continuous direct sum of two $Df$-invariant subbundles $TM=E^u\oplus E^s$ such that $Df$ uniformly expands the length of vectors in $E^u$ and uniformly contracts the length of vectors in $E^s$ (see subsection \ref{subsec:anosov_diffeomorphisms} for a more precise description). We refer to $E^u$ as the \emph{unstable bundle} and $E^s$ as the \emph{stable bundle} associated to $f$.
An important feature of Anosov diffeomorphisms is their structural stability. This means that there exists a $C^1$ neighborhood $\mc{U}$ of $f$ such that if $g\in \mc{U}$ then there exists a homeomorphism $h$ such that $h g h^{-1}=f$. The map $h$ is called a conjugacy and $f$ and $g$ are said to be conjugate. See \autocite[Sec. 2.3]{hasselblatt1995introduction} for more concerning structural stability.

In this paper, we study a rigidity phenomenon concerning conjugacies between two Anosov diffeomorphisms. It is well known that a conjugacy between two Anosov diffeomorphisms is necessarily H\"older continuous, and, in general, no more than H\"older continuity can be expected. 
If a conjugacy between two maps is $C^1$, then the maps are said to be $C^1$ conjugate.
Maps that are $C^1$ conjugate have many common features and so there are natural reasons why two diffeomorphisms cannot be $C^1$ conjugate. In our work, we consider two obvious obstructions to the existence of a $C^1$ conjugacy. Our main result is to show, in a particular setting, that if there is not an obvious obstruction, then there is indeed a $C^1$ conjugacy between a map and its perturbation.

To begin, we will describe the most elementary obstruction to the existence of a conjugacy: the periodic data. Suppose that $f$ and $g$ are two diffeomorphisms that are $C^1$ conjugate by a conjugacy $h$, so that $f=hgh^{-1}$. If $p$ is a periodic point of $f$ of period $n$, then $h(p)$ is a periodic point of $g$ of period $n$. By differentiating, we see that 
\[
D_pf^n=D_{h(p)}hD_{h(p)}g^n(D_{h(p)}h)^{-1}.
\]
Consequently, we see that the differentials of $f^n$ at $p$ and $g^n$ at $h(p)$ are conjugate.
Given two diffeomorphisms $f$ and $g$ and a conjugacy $h$, we say that $f$ and $g$ have the same \emph{periodic data with respect to }$h$ if for each periodic point $p$, if $p$ has period $n$, then $D_pf^n$ and $D_{h(p)}g^n$ are conjugate as linear maps. The previous discussion shows that for two diffeomorphisms to be $C^1$ conjugate, it is necessary for them to have the same periodic data.

We say that a map is $C^{1+}$ when the map is $C^1$ and its derivative is $\theta$-H\"older continuous for some $\theta>0$. We write $\Diff^{1+}_{\vol}(M)$ for the set of volume-preserving diffeomorphisms of $M$ that are $C^{1+}$. 
We are now able to introduce one of the two kinds of rigidity that we will study. 
\begin{defn}
We say that an Anosov diffeomorphism $f$ is \emph{locally periodic data rigid} if there exists a $C^1$ neighborhood $\mc{U}$ of $f$ inside of $\Diff^{1+}(M)$ such that if $g\in \mc{U}$ and $f$ and $g$ have the same periodic data with respect to a conjugacy $h$, then $h$ is $C^{1+}$.
\end{defn}

This type of rigidity has been well studied in the case of Anosov automorphisms of tori. Such an automorphism is defined by the natural action of an element of $\SL(n,\Z)$ on the torus $\R^n/\Z^n$. See, for instance, \autocite{de1987invariants} and \autocite{de1988invariants} for some early results in the case of automorphisms of low dimensional tori. More recently, this problem has been studied by Gogolev, Kalinin, and Sadovskaya.
In \autocite[Thm. 1]{gogolev2018local}, they establish local periodic data rigidity of an Anosov automorphism $L$ of a torus under the assumptions that no three eigenvalues of $L$ have the same modulus and that $L$ and $L^4$ are both irreducible. For earlier results, see also \autocite{gogolev2008smooth} and \autocite{gogolev2011local}.  In recent work, Saghin and Yang \autocite{saghin2018lyapunov} obtained several additional results on the torus.

In this paper, we obtain local periodic data rigidity results for Anosov automorphisms of nilmanifolds. Before stating this result, we briefly explain this setting and why it is the appropriate generalization. A nilmanifold is a smooth manifold obtained by the following construction. One begins with a nilpotent Lie group $N$ and a discrete subgroup $\Gamma$ such that $N/\Gamma$ is compact. The manifold $N/\Gamma$ is then known as a nilmanifold. If $L$ is an automorphism of $N$ preserving $\Gamma$, then $L$ descends to a map on the quotient $N/\Gamma$. Write $\mf{n}$ for the Lie algebra of $N$. The automorphism $L$ induces an automorphism of the Lie algebra $\mf{n}$. If the induced map on $\mf{n}$ has no eigenvalues of unit modulus, then the map on $N/\Gamma$ is an Anosov diffeomorphism. We refer to the map on the quotient $L\colon N/\Gamma\to N/\Gamma$ as an \emph{Anosov automorphism}. It is conjectured that if $f\colon M\to M$ is an Anosov diffeomorphism, then $M$ is finitely covered by a nilmanifold. Consequently, studying Anosov automorphisms on nilmanifolds is quite natural. As far as the author is aware, there are no known examples of Anosov automorphisms exhibiting periodic data rigidity on a non-toral nilmanifold. In this paper we establish the first such example.

We say that an Anosov automorphism has \emph{simple spectrum} if the magnitudes of all the eigenvalues of the induced map on $\mf{n}$ are distinct. In this paper, we restrict our study to automorphisms with simple spectrum. 
 However, even with this restriction not every automorphism exhibits local periodic data rigidity. In the toral case, a condition called \emph{irreducibility} is necessary for periodic data rigidity. We generalize the definition of irreducibility to a nilmanifold automorphism. As it turns out, irreducibility alone is insufficient to ensure periodic data rigidity. We also introduce a condition on the spectrum that we call \emph{sortedness}. Precise statements of these conditions are given later, as they rely on a more detailed development of the notion of an Anosov automorphism of a nilmanifold. For irreducibility, see section \ref{sec:irreducibility}. For sortedness, see Definition \ref{defn:sorted_spectrum}.
\begin{thm}\label{thm:periodic_rigidity}
Suppose that $L\colon N/\Gamma\to N/\Gamma$ is an Anosov automorphism with simple spectrum.  Then $L$ is locally periodic data rigid if and only if $L$ is irreducible and has sorted spectrum.
\end{thm}
The proof of sufficiency in this theorem relies on periodic approximation (Proposition \ref{prop:periodic_approximation}), which reduces the sufficiency claim to that in Theorem \ref{thm:rigidity}.
 The necessity of the condition follows from Theorem \ref{thm:non_rigid}.

As the title of this paper suggests, we also study Lyapunov spectrum rigidity. Before we can state our rigidity result in this direction, we briefly develop the necessary language. Suppose that $f$ is a diffeomorphism of a manifold $M$ of dimension $n$ preserving an ergodic invariant measure $\mu$. Then there exists a list of numbers $\lambda_1\le \cdots\le \lambda_n$ such that for $\mu$-a.e.\@ $x\in M$, and any non-zero $v\in T_xM$ there exists $1\le i\le n$ such that
\[
\lim_{n\to \infty} \frac{1}{n}\log \|D_xf^nv\|=\lambda_i.
\]
 The numbers $\lambda_i$ are referred to as the Lyapunov exponents of $f$ with respect to $\mu$. Note that some of the $\lambda_i$ may be repeated. We refer to this list with multiplicity as the \emph{Lyapunov spectrum} of $f$ with respect to $\mu$. In the case of an Anosov automorphism $L$, $\vol$ is an ergodic invariant measure. If the induced map on $\mf{n}$ has eigenvalues $\lambda_1,...,\lambda_n$, then the Lyapunov exponents of $L$ with respect to volume are the numbers $\log \abs{\lambda_i}$ for $1\le i\le n$. Hence when we restrict to $L$ with simple spectrum, the Lyapunov exponents of $L$ with respect to volume are all distinct. 
For more information, see \autocite[Supplement]{hasselblatt1995introduction}.

There is a close relationship between Lyapunov exponents and periodic data. The work of Kalinin in \autocite{kalinin2011livsic} contains a very useful result relating periodic data to Lyapunov exponents. We will not recapitulate this result in full, but instead state a conclusion that follows from it. Suppose that $f$ is an Anosov diffeomorphism of a nilmanifold. Kalinin establishes the following for the Lyapunov exponents of measures invariant under $f$. Suppose that $\mu$ is an ergodic invariant measure and that $\chi_1\le \cdots \le \chi_n$ are the Lyapunov exponents, listed with multiplicity, of $f$ with respect to $\mu$. For any periodic point $p$ there is a natural ergodic invariant measure supported on the orbit of $p$, namely, the uniform measure. Write the Lyapunov exponents of this measure with multiplicity as $\chi_1^{(p)}\le \cdots \le \chi_n^{(p)}$. What Kalinin shows is that for every $\epsilon>0$, there exists a point $p$, so that for $1\le i\le n$, $\abs{\chi_i-\chi_i^{(p)}}<\epsilon$ \autocite[Thm. 1.4]{kalinin2011livsic}. In this sense the Lyapunov exponents of $\mu$ are approximated by the Lyapunov exponents at a periodic point. If an Anosov diffeomorphism has the same periodic data as a linear example, then every periodic point has the same Lyapunov exponents. Consequently, we deduce the following proposition, which is immediate from Kalinin's work.

\begin{prop}\label{prop:periodic_approximation}
(Periodic Approximation) Suppose that $f$ is an Anosov diffeomorphism with the same periodic data as an Anosov automorphism $L$. Then the Lyapunov exponents of every ergodic invariant measure of $f$ coincide with those of $L$.
\end{prop}

We now introduce a notion of local rigidity that pertains to the volume Lyapunov spectrum. 
\begin{defn}\label{def:lyapunov_spectrum_rigid}
Suppose that $L\in \Diff^{1+}_{\vol}(M)$ is an Anosov automorphism.
We say that $L$ is \emph{locally Lyapunov spectrum rigid} if there exists a $C^1$ neighborhood $\mc{U}$ of $L$ in $\Diff^{1+}_{\vol}(M)$ such that if $g\in \mc{U}$, and the Lyapunov spectrum of $g$ with respect to volume is equal to the Lyapunov spectrum of $L$ with respect to volume, then $g$ is $C^{1+}$ conjugate to $L$.
\end{defn}
There are dynamical systems other than Anosov automorphisms that exhibit Lyapunov spectrum rigidity. For instance, Butler \autocite{butler2017characterizing} recently showed that closed locally symmetric spaces of negative curvature are characterized either in terms of the Lyapunov exponents of their geodesic flow or the periodic data of their geodesic flow. While Butler's result also concerns Lyapunov spectrum rigidity, his approach is quite different from the approach in this paper.

A $C^{1+}$ Anosov diffeomorphism with the same periodic data as an Anosov automorphism preserves a $C^{1+}$ volume \autocite[Thm 19.2.5]{katok1997introduction}.
Thus, by Proposition \ref{prop:periodic_approximation}, we see that in order to show that an Anosov automorphism is locally periodic data rigid that it suffices to show that the automorphism is locally Lyapunov spectrum rigid. 
Obtaining local Lyapunov spectrum rigidity is our main result.
\begin{thm}\label{thm:rigidity}
Suppose that $L\colon N/\Gamma\to N/\Gamma$ is an Anosov automorphism with simple Lyapunov spectrum. Then $L$ is locally Lyapunov spectrum rigid if and only if $L$ is irreducible and has sorted spectrum.
\end{thm}
This theorem follows from the combination of Theorem \ref{thm:sufficiency} and Theorem \ref{thm:non_rigid}, each of which shows one direction of the equivalence. We show that the theorem is not vacuous by constructing explicit examples of such automorphisms in section \ref{sec:examples}. By periodic approximation (Proposition \ref{prop:periodic_approximation}), Theorem \ref{thm:rigidity} also shows the sufficiency of the condition in Theorem \ref{thm:periodic_rigidity}. We believe that this theorem establishes the first known instance of Lyapunov spectrum rigidity for an Anosov automorphism that is not on a torus. For previous work on Lyapunov spectrum rigidity in the torus setting, see \autocite{saghin2018lyapunov} as well as
\autocite{gogolev2018local}.

In the nilmanifold case, there are several new complications that arise in the study of the local rigidity of Anosov automorphisms.
 One major complication is that certain weak foliations that exist in the toral case do not exist in the nilmanifold case because the distributions that define them are not integrable. Another difficulty is that the foliations that arise in our proof are not minimal, whereas in the work of Gogolev, Kalinin, and Sadovskaya, all foliations arising in their proof are minimal by assumption \autocite[Thm. 1.1]{gogolev2018local}. This work provides some of the first examples of rigidity in a setting without an abunance of minimal dynamical foliations. A final important difficulty is that in $\R^n$ a geodesic is a line, i.e. minimizes distance between its points. In a nilpotent Lie group, geodesics may take very inefficient paths at large scale. This difficulty is overcome by working directly within unstable manifolds and using their coarse geometry. This approach seems to be a novelty.

A more technical difference between our result and previous results is that we allow the perturbation to have slightly lower regularity. We assume that the perturbation is $C^{1+}$ whereas in \autocite{gogolev2018local} the perturbation is assumed to be $C^2$. This seems to be the lowest that the regularity  of $f$ can be lowered using the current approach. See Subsection \ref{subsec:regularity_criteria} for a more detailed discussion.

\

\textbf{Acknowledgements.} The author thanks Aaron Brown, Clark Butler, and Amie Wilkinson for their critical comments on earlier versions of this paper. The author thanks Andrey Gogolev and Michael Neaton for helpful discussions. The author also thanks Meg Doucette and Liam Mazurowski for their help improving the readability of the paper.

\subsection{Sketch of proof of Theorem \ref{thm:rigidity}}

Our proof of the sufficiency of the condition in Theorem \ref{thm:rigidity} follows the approach taken in \autocite{gogolev2018local}. We construct a neighborhood $\mc{U}$ of $L$  in the following way. By Theorem \ref{thm:mather_spectrum_perturbation}, we may choose the neighborhood $\mc{U}$ so that if $f\in \mc{U}$ then the unstable bundle $E^u=E^{u,f}$ for $f$ splits into the direct sum of one-dimensional H\"older continuous $Df$-invariant subbundles $E_i^{u,f}$.  We index these subbundles so that if $i<j$ then the expansion properties of $Df$ acting on $E_i^{u,f}$ are weaker than those of $Df$ acting on $E_j^{u,f}$. We call the $E_i^u$ distribution the $i$th unstable distribution. For each $1\le j\le \dim E^u$, the distribution $\oplus_{i\ge j} E_i^{u,f}$ uniquely integrates to a foliation, which we call the $i$th strong unstable foliation and denote by $\mc{S}_i^{u,f}$. We use this superscript notation analogously for other objects depending on the map $f$. Note that the indices $i$ are arranged so that $\mc{S}_1^{u,f}$ is the full unstable foliation and the dimension of a leaf of $\mc{S}_i^{u,f}$ decreases as $i$ increases. This construction is standard and recalled in Proposition \ref{prop:strong_is_uniformly_immersed}. We say that a conjugacy $h$ \emph{intertwines} two foliations $\mc{F}$ and $\mc{G}$ if $h(\mc{F}(x))=\mc{G}(h(x))$ for each $x$. Here and elsewhere, $\mc{F}(x)$ is the leaf of the foliation $\mc{F}$ through the point $x$.

The proof of the theorem is by an inductive argument. The core claim in the induction is that if $h$ is a conjugacy and $h$ intertwines $\mc{S}_i^{u,f}$ and $\mc{S}_i^{u,L}$, then $h$ intertwines $\mc{S}_{i+1}^{u,f}$ and $\mc{S}_{i+1}^{u,L}$. To prove this claim, we construct a one-dimensional foliation $\mc{W}_i^{u,f}$ that restricted to an $\mc{S}_i^{u,f}$ leaf has a global product structure with the $\mc{S}_{i+1}^{u,f}$ foliation (see subsection \ref{subsec:foliations} for the definition of this term). The foliation $\mc{W}_i^{u,f}$ is tangent to the $E_i^{u,f}$ distribution, and so one thinks of the $\mc{W}_i^{u,f}$ foliation as the weakest part of the $\mc{S}_i^{u,f}$ foliation. Constructing the $\mc{W}_i^{u,f}$ foliation is involved and uses a detailed study of the coarse geometry of nilpotent Lie groups, which comprises section \ref{sec:coarse_geometry}. The proof also uses the result that $h$ induces a quasi-isometry on leaves of the $\mc{S}_i^{u,f}$ foliation (see Corollary \ref{cor:QI_foliations}). In this argument, $L$ having sorted spectrum plays a crucial role.

As a byproduct of the construction of $\mc{W}_i^{u,f}$ in Proposition \ref{prop:slowtoslow}, we also obtain that $h$ intertwines the $\mc{W}_i^{u,f}$ foliation with an algebraically defined and analytic foliation $\mc{W}_i^{u,L}$. We pull back the $\mc{S}_{i+1}^{u,f}$ foliation by $h$ to obtain a foliation $\mc{F}$. We then study the holonomies of $\mc{F}$ between leaves of the $\mc{W}_i^{u,L}$ foliation. Surprisingly, by Lemma \ref{lem:F_is_isometric}, these holonomies are isometries. Via an additional argument appearing in section \ref{sec:foliations_with_isometric_isometries}, we show that the foliation $\mc{F}$ is equal to the foliation $\mc{S}_{i+1}^{u,L}$. This additional argument also requires a detailed study of the geometry of nilmanifolds. In this argument, the assumption of irreducibility plays its most important role. One finally concludes that $h$ is $C^{1+}$ by studying the disintegration of volume along the $\mc{W}_i^{u,f}$ foliations; the key idea used is discussed in Subsection \ref{subsec:regularity_criteria}.

The proof of necessity of the condition in Theorem \ref{thm:rigidity} is a systematic procedure for perturbing an automorphism $L$ should either sortedness of the Lyapunov spectrum or irreducibility fail. The general idea is to shear a fast direction into a slower one. This argument appears in section \ref{sec:necessity}.

\section{Preliminaries}

\subsection{Anosov diffeomorphisms}\label{subsec:anosov_diffeomorphisms}

We say that an automorphism of a finite-dimensional real vector space is \emph{hyperbolic} if it has no eigenvalues of modulus one.
 A hyperbolic linear map $A\colon \R^n\to \R^n$ decomposes $\R^n$ into the direct sum of two subspaces: a stable subspace, $E^s$, on which $A$ is a contraction, and an unstable subspace, $E^u$, on which $A^{-1}$ is a contraction.
  We say that an eigenvector $v$ of $A$ is stable or unstable according to which subspace it lies in. In this paper, we study diffeomorphisms whose differentials satisfy an analogous property.
A diffeomorphism $f$ of a compact manifold $M$ is \emph{Anosov} if there exists a continuous splitting of $TM$ into the direct sum of two $Df$-invariant subbundles $E^{s,f}$ and $E^{u,f}$, a Riemannian metric on $M$, and $\lambda>1$ such that, for any $x\in M$,
\[
\|D_xf\mid_{E^{s,f}}\|< \lambda^{-1}<1<\lambda<\|D_xf^{-1}\mid_{E^{u,f}}\|^{-1}.
\]
We refer to the distributions $E^{s,f}$ and $E^{u,f}$ as the stable and unstable distributions of $f$, respectively.

A nilmanifold $N/\Gamma$ is a smooth manifold obtained as the quotient of a nilpotent Lie group $N$ by a cocompact lattice $\Gamma$. For information on nilmanifolds and nilpotent Lie groups, see \autocite{raghunathan1972discrete}. A right-invariant metric on $N$ descends to a metric on $N/\Gamma$ and makes the projection $\pi\colon N\to N/\Gamma$ a local isometry.

\begin{defn}\label{defn:anosov_automorphism}
We say that a map $L\colon N/\Gamma\to N/\Gamma$ of a nilmanifold $N/\Gamma$ is an \emph{Anosov automorphism} if the natural lift of $L$ to $N$ is an automorphism of $N$ and the differential of the lift at $e\in N$ is hyperbolic. In an abuse of notation, we may use $L$ to refer to both the map on $N$ and $N/\Gamma$.
\end{defn}

An Anosov automorphism is an Anosov diffeomorphism. On a nilmanifold, every Anosov diffeomorphism is topologically conjugate to an Anosov automorphism. In the case of nilmanifolds, this is a theorem of Franks \autocite{franks1969anosov} and Manning \autocite{manning1974there}. An infranilmanifold is a manifold that is finitely covered by a nilmanifold. Some infranilmanifolds also admit Anosov diffeomorphisms due to a refinement of the construction in the nilmanifold case. 

In this paper, we consider the regularity of a conjugacy $h$ between an Anosov automorphism $L$ and an Anosov diffeomorphism $f$. In general, there may be infinitely many such conjugacies. However, all conjugacies between $f$ and $L$ have the same regularity. For a discussion of all possible conjugacies, see \autocite{kalinin2010errata}.

\begin{prop}\label{prop:same_regularity}
Suppose that $f$ is an Anosov diffeomorphism and $L$ is an Anosov automorphism. If, for some $k\in \mathbb{N}$ and $\theta\in [0,1)$, there exists a $C^{k+\theta}$ conjugacy between $f$ and $L$, then every other conjugacy between $f$ and $L$ is $C^{k+\theta}$.
\end{prop}
\begin{proof}
 Conze and Marcuard \autocite[Thm. 1]{conze1970conjugaison} proved that if $\gamma$ is a homeomorphism and $A$ is an Anosov automorphism such that $\gamma A\gamma^{-1}=A$,  then $\gamma$ is an affine transformation.  
Note that if $hf=Lh$ and $h'f=Lh'$, then $h'h^{-1}Lhh'^{-1}=L$ and $hh'^{-1}Lh'h^{-1}=L$. Applying the result of Conze and Marcuard to $h'h^{-1}$ and $h'^{-1}h$, we see that both are affine and hence $C^{\infty}$. This implies that $h'$ is as regular as $h$. 
\end{proof}

\subsubsection{Mather spectrum}

We say that a diffeomorphism has \emph{simple Mather spectrum} if there exists a continuous splitting of $TM$ into one-dimensional $Df$-invariant subbundles $E_i$, i.e.  $TM=E_1\oplus \cdots \oplus E_{\dim M}$, and there exist constants $a_i<b_i$ such that $b_i<a_{i+1}$ and a Riemannian metric on $M$ such that if $v\in E_i$,
\begin{equation}\label{eq:mather_growth_estimate}
a_i\|v\|\le \|D_xf v\|\le b_i\|v\|.
\end{equation}
Note that this is different from the usual definition of Mather spectrum but is equivalent. See, for instance, the introduction to \autocite{jiang1995integrability}. For a diffeomorphism $f$ with simple Mather spectrum, we say that the spectrum is contained in $[a_i,b_i]$-rings in accordance with the constants $a_i$ and $b_i$ in equation \eqref{eq:mather_growth_estimate}.

We will use the following standard result when making perturbations. See, for example, \autocite[Thm. A]{jiang1995integrability}.
\begin{thm}\label{thm:mather_spectrum_perturbation}
 Suppose that $f$ has simple Mather spectrum contained in $[a_i,b_i]$-rings. For every $\epsilon>0$, there exists a neighborhood $\mc{U}$ of $f$ in $\Diff^1(M)$ such that the Mather spectrum of any diffeomorphism $g\in \mc{U}$ is contained in $[a_i',b_i']$ rings with $0<a_i-a_i'<\epsilon$ and $0<b_i'-b_i<\epsilon$, $1\le i\le \dim M$, and the corresponding splitting $TM=\oplus_i E_i^g$ satisfies $\rho(E_i^f,E_i^g)\le \epsilon$, where $\rho$ is a metric on the Grassmanian of $1$-planes induced by the Riemannian metric on $M$.
\end{thm}

Note that simple Mather spectrum gives an estimate, equation \eqref{eq:mather_growth_estimate}, that is uniform over all of $M$.
 In the sequel, we  work with the decomposition of $TM$ into one-dimensional subbundles having the properties given in Theorem \ref{thm:mather_spectrum_perturbation}. In the case of simple Mather spectrum, there is a splitting of the unstable bundle into continuous one-dimensional $Df$-invariant subbundles. In the notation of the theorem, we write:
\begin{equation}
E^{u,f}=\bigoplus_{i=1}^{\dim E^u} E^{u,f}_i,
\end{equation}
and order the subspaces so that $i<j$ implies that the maximum expansion $b_i$ of $E^{u,f}_i$ is smaller than the minimum expansion $a_j$ of $E^{u,f}_j$. We refer to $E_i^{u,f}$ as the \emph{$i$th unstable subspace}.

\subsection{Dynamical foliations for Anosov automorphisms}

We describe algebraically the stable and unstable manifolds of an Anosov automorphism. Let $L$ be an automorphism of a nilpotent Lie group $N$.
For $g\in N$, write $R_g$ for right multiplication by $g$. We trivialize the tangent bundle as $N\times T_eN=N\times \mf{n}$ by the map that sends
\[
(g,v)\mapsto (R_g)_*v\in T_{g}N.
\]
The above map $N\times T_eN\to TN$ is the usual trivialization of the tangent bundle of $N$ on the right. For a diffeomorphism $f\colon N\to N$, the differential of $f$ in this trivialization is
\[
(g,v)\mapsto (f(g),(D_{f(g)}R_{f(g)^{-1}})(D_gf)(D_eR_g)v).
\]
In the case that $f=R_g$, the induced map in the trivialization is
\[
(h,v)\mapsto (R_gh,v).
\]
These constructions allow a precise study of the dynamics of an automorphism of $N$ and of its induced map on a quotient $N/\Gamma$ for an invariant lattice $\Gamma$. The results of such a study are summarized in the following proposition.

\begin{prop}\label{prop:dynamical_foliations_for_nilmanifold_automorphisms}
(Dynamical foliations for nilmanifold automorphisms). Suppose that $L\colon N\to N$ is an automorphism of a nilpotent Lie group $N$ such that $D_eL\colon \mf{n}\to \mf{n}$ has simple real spectrum with no eigenvalues of modulus $1$. Index the eigenvalues of $D_eL$ of modulus greater than one so that
\[
1<\abs{\lambda_1^u}<\cdots<\abs{\lambda_k^u},
\]
and write $E_{\lambda_i^u}$ for the corresponding eigenspace. 
\begin{enumerate}
\item

For $1\le i\le k$, the ``strong" linear subspace $\mf{s}_{i}^u\coloneqq \oplus_{j\ge i} E_{\lambda_j^u}$ is a subalgebra of $\mf{n}$ tangent to a subgroup $S_i^u$. Similarly, the ``weak" subspace $E_{\lambda_i^u}$ is tangent to a subgroup $W_i^u$.  There exists a right-invariant foliation of $N$ obtained from the right translates of these subgroups. The leaf through $x\in N$ of these foliations is equal to $S_i^ux$ and $W_i^ux$, respectively.
\item

Suppose that $\Gamma$ is a lattice such that $L(\Gamma)=\Gamma$. Then $L$ descends to a map $N/\Gamma\to N/\Gamma$. Moreover, for each $1\le i\le k$, the foliations of $N$ with leaves $W_i^ux$ and $S_i^ux$ descend to $L$-invariant foliations of $N/\Gamma$. We refer to these foliations as the $\mc{W}_i^{u,L}$ and $\mc{S}_i^{u,L}$ foliations. For $x\Gamma\in N/\Gamma$, $\mc{W}_i^{u,L}(x\Gamma)=W_i^{u,L}x\Gamma$ and $\mc{S}_i^{u,L}(x\Gamma)=S_i^{u,L}x\Gamma$.
\item

The right-invariant distribution defined by $E_{\lambda_{i}^u}$ at the identity  projects to a distribution on $N/\Gamma$. The leaves of the $\mc{W}_i^{u,L}$ foliation are characterized by their tangency to this distribution. Similarly, the leaves of the $\mc{S}_i^{u,L}$ foliation are characterized by their tangency to the projection of the right-invariant distribution arising from the subalgebra $\mf{s}_i^u=\oplus_{j\ge i} E_{\lambda_j^u}\subseteq \mf{n}$.  In particular, these distributions are uniquely integrable.
\item

Finally, any right-invariant framing of $E_{\lambda_i^u}$ on $N$ projects to a framing of the corresponding distribution on $N/\Gamma$ via the differential of $\pi\colon N\to N/\Gamma$. Fixing such a framing, we identify the action of the differential of $L$ restricted to the $W_i^ux\Gamma$ foliation with the action of $L$ on the right-invariant framing of $E_{\lambda_i^u}$ on $N$, which itself is identified with the action of $D_eL$ on $\mf{n}$ via the trivialization. The action on $E_{\lambda_i^u}\subset \mf{n}$ is multiplication by $\lambda_i^u$. So, with respect to this framing of the projection of $E_{\lambda_i^u}$ to $N/\Gamma$, the differential of $L$ is multiplication by $\lambda_i^u$. Similar considerations apply to $S_i^u$.
\end{enumerate}
\end{prop}

\begin{rem}
In the sequel we use the notation $\mc{W}_i^{u,L}$ and $\mc{S}_i^{u,L}$ to denote these ``weak" and ``strong" foliations on both $N$ and $N/\Gamma$. Given the context in which this notation appears, this should cause no confusion.
\end{rem}

We now outline the proof of the above proposition.

\begin{proof}[Proof of Proposition \ref{prop:dynamical_foliations_for_nilmanifold_automorphisms}.]
Fixing a metric on $T_eN$, we obtain a right-invariant metric on $N$. The automorphism $L$ has differential $D_eL\colon \mf{n}\to \mf{n}$ at $e$. With respect to the trivialization of the tangent bundle on the right, the map $DL\colon N\times T_eN \to N\times T_eN$ is
\[
(g,v)\mapsto (L(g),(D_{L(g)}R_{L(g)^{-1}})(D_gL)(D_eR_g)v).
\]
(Note that $L(g)$ is the image of $g$ under the automorphism $L$ and not left translation.)
In the above expression, the composition of differentials applied to $v$ is the differential of the map $h\mapsto L(hg)L(g)^{-1}=L(h)$ at $e$. As this map is equal to $L$, we see that the map being applied to $v$ is $D_eL$. Consequently, with respect to the trivialization of $TN$ by right translation, $DL$ has the expression
\[
D_g(g,v)=(L(g),(D_eL)v).
\]

By assumption, $\mf{n}$ splits into one-dimensional eigenspaces of $DL$, so that $\mf{n}=\oplus_{\lambda} E_{\lambda}$. By pushing this splitting into eigenspaces through the trivialization, we obtain a right-invariant, $DL$-invariant splitting of $TN$.
A right-invariant metric makes $L$ Anosov on $N$. In particular, any estimate we have on the growth of norm of vectors in $E_{\lambda}\subset \mf{n}$ now holds globally for the splitting on $N$.

\begin{enumerate}
\item

Consider $D_eL\colon \mf{n}\to \mf{n}$ in this way. As we previously assumed, there is a splitting of $\mf{n}$ into one-dimensional eigenspaces of $D_eL$ that we write as $\mf{n}=\oplus_{\lambda} E_{\lambda}$ where $\lambda\in \R$. Observe that if $v\in E_{\lambda}$ and $w\in E_{\lambda'}$ then $DL([v,w])=[DL v,DL w]=\lambda\lambda'[v,w]$. 

As the bracket of eigenvectors is either an eigenvector or zero, for a fixed $i$, $\oplus_{\abs{\lambda}\ge \abs{\lambda_i^u}} E_{\lambda}$ is a subalgebra of $\mf{n}$ and hence is tangent to an analytic subgroup of $N$. We write $S_i^u$ for the analytic subgroup tangent to $\oplus_{\abs{\lambda}\ge \abs{\lambda_i^u}} E_{\lambda}$ at $e$.
 We also consider the one-dimensional subspaces $E_{\lambda}$, which are subalgebras as they are one-dimensional. We write $W_i^u$ for the analytic subgroup tangent to $E_{\lambda_i^u}$ at $e$.

We write $\mf{n}^s$ for the subspace of $\mf{n}$ spanned by eigenvectors with eigenvalue of modulus less than one and $\mf{n}^u$ for the subspace of $\mf{n}$ spanned by eigenvectors with eigenvalue of modulus greater than one. In this case we write $N^u$ and $N^s$ for the corresponding analytic subgroups. Note that $N^u=S_1^u$.

Suppose that $E\subset \mf{n}$ is a $D_eL$-invariant subspace. If $E$ is a subalgebra of $\mf{n}$, then $E$ is tangent to an analytic subgroup of $N$, which we will call $N_E$. The right translates of $N_E$ are tangent to the right translates of $E$. Hence any properties of the differential of $L$ on $E$ hold at every point of $N_Ex$ for all $x\in N$, with respect to a right-invariant metric. Applying this reasoning to the subalgebras $\oplus_{\abs{\lambda}\ge \abs{\lambda_i^u}} E_{\lambda}$ and $E_{\lambda}$ establishes the first part of the proposition. 
\item

Consider a lattice $\Gamma\subset N$ such that $L(\Gamma)=\Gamma$. Note that $N/\Gamma$ is a compact manifold, see \autocite[Thm. II.2.1]{raghunathan1972discrete}, and, as the metric we chose on $N$ is right-invariant, the quotient map $\pi\colon N\to N/\Gamma$ is a local isometry. Moreover, right-invariant structures on $N$ descend to structures on $N/\Gamma$. As the foliations with leaves $W_i^ux$ and $S_i^ux$ are right-invariant, they descend to foliations on $N/\Gamma$ via the projection. We refer to these foliations as $\mc{W}_i^{u,L}$ and $\mc{S}_i^{u,f}$ respectively. 

\item

The tangent distribution to the $W_i^ux$ foliation and $S_i^ux$ foliations are right-invariant, and hence descend to distributions tangent to the foliations $\mc{W}_i^{u,L}$ and $\mc{S}_i^{u,L}$. An integrable smooth distribution is uniquely integrable, so we see that the $\mc{W}_i^{u,L}$ and $\mc{S}_i^{u,L}$ distributions are characterized by their tangent distributions.

\item

The final part is immediate. A fixed framing of the invariant splitting of $\mf{n}$ extends to a right-invariant framing of $N$ that projects to a framing of $N/\Gamma$. As $\pi\colon N\to N/\Gamma$ is a local isometry and respects the action of $L$, we see that the action on a frame may be computed by lifting the frame and then projecting back to $N/\Gamma$. Consequently, any estimates on the framing in $N$ give estimates on the framing in $N/\Gamma$.

\end{enumerate}

\end{proof}

\subsection{Automorphisms of nilmanifolds}\label{subsec:automorphisms}

In this section, we give an explicit description of the eigenspace decomposition associated to an automorphism of a real nilpotent Lie algebra with sorted simple spectrum.

We write $\mathbb{N}=\{1,2,\ldots\}$. An $\mathbb{N}$-grading of a Lie algebra $\mf{g}$ is a direct sum decomposition $\mf{g}=\oplus_{i\in \mathbb{N}}\cgrad{\mf{g}}{i}$ such that $[\cgrad{\mf{g}}{{i}},\cgrad{\mf{g}}{{j}}]=\cgrad{\mf{g}}{i+j}$. We say that a $\mathbb{N}$-grading of a nilpotent Lie algebra is \emph{Carnot} if $\cgrad{{\mf{g}}}{1}$ generates $\mf{g}$ as a Lie algebra, i.e. $\mf{g}$ is the smallest subalgebra of itself containing $\cgrad{\mf{g}}{1}$. 
For a nilpotent Lie algebra $\mf{n}$, we write $\lcc{\mf{n}}{i}$ for the $i$th term in the lower central series of $\mf{n}$. Recall that, by definition, $\lcc{\mf{n}}{1}=\mf{n}$ and $\lcc{\mf{n}}{i+1}=[\mf{n},\lcc{\mf{n}}{i}]$. 
If $\mf{n}$ is a nilpotent Lie algebra equipped with an automorphism $L$ that admits a real eigenbasis, we define the $L$\emph{-grading} of $\mf{n}$ in the following way: define $\cgrad{{\mf{n}}}{i}$ to equal to the linear subspace generated by the eigenvectors of $L$ acting on $\lcc{\mf{n}}{i}$ that do not lie in $\lcc{\mf{n}}{i+1}$. 
While the notation $\cgrad{\mf{n}}{i}$ for the $L$-grading does not demonstrate the dependence on $L$, we only ever consider one automorphism at a time and this should not cause confusion.

We are now able to define sorted spectrum.

\begin{defn}\label{defn:sorted_spectrum}
Suppose $L\colon \mf{n}\to \mf{n}$ is an automorphism of a nilpotent Lie algebra $\mf{n}$. We say that $L$ has \emph{sorted unstable spectrum} if, for any $j>k$ and any two unstable eigenvectors $v$ and $w$ with eigenvalues $\lambda_v$ and $\lambda_w$, if $v\in \cgrad{\mf{n}}{k}$ and $w\in \cgrad{\mf{n}}{j}$, then $\abs{\lambda_w}>\abs{\lambda_v}$. We say that $L$ has \emph{sorted stable spectrum} if $L^{-1}$ has sorted unstable spectrum. We say that an Anosov automorphism of $N/\Gamma$ has sorted spectrum if the induced map on $\mf{n}$ has sorted stable and unstable spectrum.
\end{defn}

\begin{prop}\label{prop:Lgradingcarnot}
Suppose that $L$ is an automorphism of a real nilpotent Lie algebra $\mf{n}$ with a real eigenbasis. Then the $L$-grading of $\mf{n}$ is Carnot. Further, if $v\in \cgrad{\mf{n}}{i}$ is an eigenvector, then there exist $v_1,\ldots,v_i\in \cgrad{\mf{n}}{1}$ such that $v=[v_1,[v_2,\ldots,[v_{i-1},v_i]\ldots]]$.
\end{prop}

\begin{proof}
We begin by showing that for $i\ge 2$, $\cgrad{\mf{n}}{i}=[\cgrad{\mf{n}}{1},\cgrad{\mf{n}}{i-1}]$.
So, suppose $i\ge 2$. Then a basis for $[\cgrad{\mf{n}}{1},\cgrad{\mf{n}}{{i-1}}]$ is obtained as a subset of the brackets of a basis for $\cgrad{\mf{n}}{1}$ with a basis for $\cgrad{\mf{n}}{i-1}$. The bracket of eigenvectors of $L$ is another eigenvector of $L$. Consequently, $[\cgrad{\mf{n}}{1},\cgrad{\mf{n}}{i-i}]$ has a basis comprised of eigenvectors of $L$. Further, any bracket is in $\lcc{\mf{n}}{i}$. Thus $[\cgrad{\mf{n}}{1},\cgrad{\mf{n}}{i-1}]\subseteq \cgrad{\mf{n}}{i}$ by definition of $\cgrad{\mf{n}}{i}$. Consequently, we are done once we show that $\cgrad{\mf{n}}{i}\subseteq [\cgrad{\mf{n}}{1},\cgrad{\mf{n}}{i-1}]$.

Suppose that $v\in \cgrad{\mf{n}}{i}$. Then there exist $r_k\in \mf{n}$ and $s_k\in \lcc{\mf{n}}{i-1}$ such that $v=\sum_k [r_k,s_k]$. Write $r_k=r_{k,1}+r_{k,2}$ where $r_{k,1}\in \cgrad{\mf{n}}{1}$ and $r_{k,2}\in \lcc{\mf{n}}{2}$. Similarly, write $s=s_{k,1}+s_{k,2}$ where $s_{k,1}\in \cgrad{\mf{n}}{i-1}$ and $s_{k,2}\in \lcc{\mf{n}}{i}$. Then 
\[
v=\sum_k [r_k,s_k]=\sum_k [r_{k,1}+r_{k,2},s_{k,1}+s_{k,2}]=\sum_{k}[r_{k,1},s_{k,1}]+[r_{k,1},s_{k,2}]+[r_{k,2},s_{k,1}]+[r_{k,2},s_{k,2}].
\]
The second and third terms inside the sum are in $\lcc{\mf{n}}{i+1}$, and the fourth term is in $\lcc{\mf{n}}{i+2}$. Thus as $v\in \cgrad{\mf{n}}{i}$, $[r_{k,1},s_{k,1}]\in \cgrad{\mf{n}}{i}$, and $\lcc{\mf{n}}{i}=\cgrad{\mf{n}}{i}\oplus \lcc{\mf{n}}{i+1}$, the final three terms are $0$. So, $v=\sum_k [r_{k,1},s_{k,1}]$, and so $\cgrad{\mf{n}}{i}=[\cgrad{\mf{n}}{1},\cgrad{\mf{n}}{i-1}]$.

We now show the second claim in the theorem: an eigenvector is the bracket of eigenvectors in $\cgrad{\mf{n}}{1}$. The second claim immediately implies that the grading is Carnot. We give a proof by induction. Suppose that the claim holds for $i$. Let $V=\{v_j\}$ be an eigenbasis for $\cgrad{\mf{n}}{1}$ and let $W=\{w_k\}$ be an eigenbasis for $\cgrad{\mf{n}}{i}$. By hypothesis, each $w_k$ is a repeated bracket of the vectors $v_j$. Observe that
\[
\cgrad{\mf{n}}{i+1}=[\cgrad{\mf{n}}{1},\cgrad{\mf{n}}{i}]=\spann\{[r,s]: r\in \cgrad{\mf{n}}{1}, s\in \cgrad{\mf{n}}{i}\}=\spann\{[v_j,w_k]: v_j\in V, w_k\in W\}.
\]
Thus as $\cgrad{\mf{n}}{i+1}$ is spanned by the vectors $[v_j,w_k]$, it has a basis consisting of vectors of the form $[v_j,w_k]$, each of which is an eigenvector. Thus every eigenvector of $L$ in $\cgrad{\mf{n}}{i+1}$ is of the form $[v_j,w_k]$, and hence is the repeated bracket of eigenvectors in $\cgrad{\mf{n}}{1}$.
\end{proof}

If an automorphism $L$ of $\mf{n}$ is hyperbolic, we write $\mf{n}^s$ and $\mf{n}^u$ for its stable and unstable subspaces. Suppose that $L$ has simple spectrum. Then $\mf{n}^s$ and $\mf{n}^u$ are subalgebras of $\mf{n}$ spanned by the stable and unstable eigenvectors of $L$ acting on $\mf{n}$, respectively. Consequently, these are each subalgebras of $\mf{n}$ as eigenvectors of $L$ bracket to other eigenvectors or to $0$.

\begin{prop}\label{prop:directsum}
Suppose that $L$ is an automorphism of a real nilpotent Lie algebra $\mf{n}$ with sorted spectrum admitting an eigenbasis, and let $\mf{n}^s$ and $\mf{n}^u$ be the stable and unstable subalgebras. Then $[\mf{n}^s,\mf{n}^u]=0$. Thus $\mf{n}=\mf{n}^s\oplus \mf{n}^u$ as a Lie algebra.
\end{prop}
\begin{proof}
Suppose not. Then there exists a stable eigenvector $v\in \mf{n}^s$ and an unstable eigenvector $w\in \mf{n}^u$ such that $[v,w]\neq 0$. Suppose that $\abs{\lambda_v}^{-1}<\abs{\lambda_w}$; if the reverse inequality holds a similar argument applies. Suppose $w\in \cgrad{\mf{n}}{i}$.  Then $[v,w]\in \lcc{\mf{n}}{i+1}$ is an unstable eigenvector with eigenvalue smaller in magnitude than $w$, contradicting that $L$ has sorted spectrum.
\end{proof}

The proof of the following is then almost immediate.

\begin{prop}
Suppose that $L$ is an automorphism of a real nilpotent Lie algebra $\mf{n}$ admitting a real eigenbasis and having sorted spectrum. Then the $L$-grading of the unstable algebra $\mf{n}^u$ is Carnot.
\end{prop}

\begin{proof}
Proposition \ref{prop:directsum} implies that $\lcc{\mf{n}}{i}=\lcc{\mf{n}^s}{i}\oplus \lcc{\mf{n}^u}{i}$. This implies immediately that $\cgrad{\mf{n}^u}{i}=\cgrad{\mf{n}}{i}\cap \mf{n}^u$ and that $\cgrad{\mf{n}^s}{i}=\cgrad{\mf{n}}{i}\cap \mf{n}^s$.  That the grading is Carnot now follows from Proposition \ref{prop:Lgradingcarnot}.
\end{proof}

From this we easily deduce:

\begin{prop}
Suppose that $L$ is an automorphism of a real nilpotent Lie algebra $\mf{n}$ admitting a real eigenbasis and that $L$ has sorted spectrum. Let $\lambda_1,\ldots,\lambda_{k}$ be the eigenvalues of the eigenvectors in $\cgrad{\mf{n}^u}{1}$. Then if $v$ is an eigenvector in $\cgrad{\mf{n}^u}{j}$, then the eigenvalue of $v$ is equal to $\lambda_{i_1}\cdots \lambda_{i_j}$ where each $i_j$ satisfies $1\le i_j\le k$. 
\end{prop}

\begin{proof}
This is immediate because the $L$-grading of $\mf{n}^u$ is Carnot; $v$ is the bracket of $j$ eigenvectors of $L$ lying in $\cgrad{\mf{n}^u}{1}$.
\end{proof}

The eigenvalues of $L$ do not accurately reflect the divergence of points in the large scale geometry of a nilpotent Lie group. We say that an automorphism $L$ of $\mf{n}$ is \emph{expanding} if $\mf{n}=\mf{n}^u$.
Let $L$ be an expanding automorphism of a nilpotent Lie algebra $\mf{n}$. For an eigenvector $v\in \cgrad{\mf{n}}{i}$, write $\lambda_v$ for the eigenvalue of $v$. We define 
\begin{equation}\label{defn:escapespeed}
\sigma_v=\abs{\lambda_v}^{1/i}.
\end{equation}
 We refer to $\sigma_v$ as the \emph{escape speed of $L$ in the direction $v$}.

\begin{cor}\label{cor:smallest_escape_rate}
Suppose that $L$ is an expanding automorphism of a nilpotent Lie algebra $\mf{n}$ and that $L$ has simple real spectrum. Let $v$ be an eigenvector associated to the smallest magnitude eigenvalue $\lambda_1$ of $L$. Then for any eigenvector $w$ such that $v$ and $w$ are linearly independent, we have $\sigma_w>\sigma_v=\abs{\lambda_1}$.
\end{cor}

\begin{proof}
By Proposition \ref{prop:Lgradingcarnot}, the $L$-grading of $\mf{n}$ is Carnot. So, for any other eigenvector $w\in \cgrad{\mf{n}}{i}$, we may write $w$ as the bracket of eigenvectors $v_1,\ldots,v_i\in \cgrad{\mf{n}}{1}$, so that the eigenvalue of $w$ is $\lambda_{v_1}\cdots\lambda_{v_i}$. By the assumption of simple spectrum, the modulus of one of the terms in this product is greater than $\abs{\lambda_{1}}$ and the modulus of each other term is at least $\abs{\lambda_1}$, so we obtain $\sigma_w=\abs{\lambda_{v_1}\cdots\lambda_{v_i}}^{1/i}>\abs{\lambda_1}=\sigma_v$.
\end{proof}

\subsection{Foliations}\label{subsec:foliations}

We now recall some notions concerning foliations. For a more detailed discussion, see \autocite{pugh1997holder}.
Let $\mc{F}$ be a foliation of a closed manifold $M$. For $p\in M$,
we may locally represent the leaf $\mc{F}(p)$ as a graph. For a normed space $F$, we write $F(\delta)$ for the closed disk of radius $\delta$ around $0$. For small $\delta>0$, there is a unique function $g(\cdot, y)\colon F_p(\delta)\to F_p^{\perp}$, where $F_p$ is a subspace of $T_pM$,  such that 
\[
\phi\colon (x,y)\mapsto \exp_p\circ (x,g(x,y))
\]
is a foliation chart.

We say that a foliation has \emph{uniformly $C^{k+\theta}$ leaves} if the map $x\mapsto g(x,y)$ is $C^{k+\theta}$ and its derivatives of order less than $k$ with respect to $x$ are continuous in $(x,y)$ and the $k$th derivative varies H\"older continuously with exponent $\theta$. 

We say that a foliation has \emph{uniformly $C^{k+\theta}$ holonomy} if the map $h\colon y\mapsto g(x,y)$ is $C^{k+\theta}$ and its derivatives of order less than $k$ with respect to $y$ depend continuously on $(x,y)$ and the $k$th derivative varies H\"older continuously with exponent $\theta$.

We say that two foliations $\mc{F}$ and $\mc{G}$ of a manifold $M$ have \emph{global product structure} if $\dim \mc{F}+\dim \mc{G}=\dim M$ and for each distinct pair $x,y\in M$ the set $\mc{F}(x)\cap \mc{G}(y)$ consists of exactly one point. Note that given a global product structure, we may identify two leaves of the $\mc{F}$ foliation by ``sliding" along the leaves of the $\mc{G}$ foliation.
We say that two subfoliations $\mc{F}$ and $\mc{G}$ of a foliation $\mc{W}$ have \emph{subordinate global product structure to} $\mc{W}$ if $\dim \mc{F}+\dim \mc{G}=\dim \mc{W}$ and the restrictions of the foliations $\mc{F}$ and $\mc{G}$ to any leaf of $\mc{W}$ give a global product structure on that leaf.

 Given two subfoliations $\mc{F}$ and $\mc{G}$ with subordinate global product structure to a foliation $\mc{W}$ and $a,b\in \mc{W}(c)$, define 
$H^{\mc{F}}_{a,b}\colon \mc{G}(a)\to \mc{G}(b)$ by $H^{\mc{F}}_{a,b}(x)=\mc{F}(x)\cap \mc{G}(b)$. We refer to $H^{\mc{F}}_{a,b}$ as the $\mc{F}$-holonomy between the leaves $\mc{G}(a)$ and $\mc{G}(b)$. Briefly we call it the \emph{$\mc{F}$-holonomy}. Note that for any $a,b\in \mc{W}(c)$, $H^{\mc{F}}_{a,b}$ is continuous. 

We will also take the opportunity to define some notation. Suppose $\mc{F}$ is a foliation with $C^1$ leaves of a Riemannian manifold $M$. The inclusion of a leaf of $\mc{F}$ into $M$ is $C^1$, and so we may pullback the Riemannian metric on $M$ to a metric on the leaf. We endow each leaf with this pullback metric. For two points $x,y$ in the same leaf of $\mc{F}$,  we define the distance $d_{\mc{F}}(x,y)$ to be the distance between $x$ and $y$ with respect to the pullback metric on $\mc{F}(x)$.

\subsubsection{Isometries of one-dimensional algebraic foliations}\label{sec:isometries}

Later, in Lemma \ref{lem:F_is_isometric}, we will obtain a foliation of $N$ whose holonomies are isometries.
For later use, we record an algebraic description of all such isometries.

\begin{prop}\label{prop:isometry_description}
Suppose that $W$ is a one-parameter subgroup of nilpotent Lie group $N$ and that $x,y\in N$. If $I\colon Wx\to Wy$ is an orientation-preserving isometry with respect to the induced Riemannian metric on $Wx$ and $Wy$, then $I$ is given by right multiplication on $N$ restricted to $Wx$. In particular, if $I$ takes $x$ to $y$, then $I$ is the restriction of
\[
z\mapsto zx^{-1}y.
\]
Every other such isometry between $Wx$ and $Wy$ is of the form 
\[
z\mapsto zx^{-1}wy
\]
for some $w\in W$.
\end{prop}

\begin{proof}

First we describe all isometries from $Wx$ to $Wx$. As right multiplication is an isometry on $W$, it restricts to an isometry on submanifolds. In particular, for $w\in W$, note that right multiplication by $x^{-1}wx$ is an isometry from $Wx$ to $Wx$. Consequently, as $Wx$ is isometric to $\R$, we see that these are all orientation-preserving isometries from $Wx$ to $Wx$.

Note now that right multiplication by $x^{-1}y$ is an isometry between $Wx$ and $Wy$ carrying $x$ to $y$. Consequently, by the previous discussion of isometries from $Wx$ to itself, any isometry is given by right multiplication by $x^{-1}y$ followed by right multiplication by $y^{-1}wy$ for some $w\in W$.
\end{proof}

\subsection{A criterion for regularity}\label{subsec:regularity_criteria}

We summarize a result in Saghin and Yang \autocite[Thm. G]{saghin2018lyapunov} that we will use in the proof of Theorem \ref{thm:sufficiency} to establish the regularity of a conjugacy along the leaves of a foliation.

Suppose that $\mc{F}$ is a continuous foliation of a compact manifold $M$ with uniformly $C^{r+\theta}$ leaves where $r\in \mathbb{N}$ and $\theta\in (0,1]$. 
Assume that $\mc{F}$ is invariant under a $C^{r+\theta}$ diffeomorphism,  $f\colon M\to M$, and that with respect to some Riemannian metric $\|\cdot\|$ on $M$ there exists $\lambda>1$ such that for all $x\in M$  $\|D_xf\mid_{T_x\mc{F}}\|\ge \lambda$. 
 We say that such a foliation is a \emph{$C^{r+\theta}$ expanding foliation} for $f$.

A particular type of absolutely continuous measure is used for detecting the regularity of a conjugacy intertwining expanding foliations. We now recall the notion of a disintegration of measures in a manner that is adapted to our context. For a detailed account, see \autocite{rokhlin1952fundamental}. Let $\mu$ be a Borel measure on a manifold $M$ and let $\mc{F}$ be a foliation on $M$. For a foliation chart $\phi\colon B_\tau^n\times B_{\perp}^{n-k}\to U\subset M$ the plaques of the chart form a partition $\mc{P}$ of the set $U$. The natural projection $\pi\colon U\to \mc{P}$ allows us to push foward the restriction of the measure $\mu$ to $U$ to a measure $\overline{\mu}$ on $\mc{P}$. We then seek a system of measures $\{\mu_{P}\}_{P\in \mc{P}}$ where each $\mu_P$ is a Borel measure on the plaque $P$. By \emph{system of measures} we mean that $\mu_{P}(P)=1$ for $\overline{\mu}$-a.e. $P$, that for a fixed continuous $f$ that the function $P\mapsto \int_P f\,d\mu_{P}$ is measurable, and further that we may express an integral over $U$ against $\mu$ as an iterated integral:
\[
\int_U f\,d\mu=\int_{\mc{P}}\int_{P} f\,d\mu_P\,d\overline{\mu}.
\]
By the work of Rokhlin, there exists such a system of measures \autocite{rokhlin1952fundamental}. Further, if we had two systems $\{\mu^\alpha_P\}_{P\in \mc{P}}$ and $\{\mu^\beta_{P}\}_{P\in \mc{P}}$, then $\mu^\alpha_P=\mu^\beta_P$ for $\overline{\mu}$-a.e. $P$. We refer to these systems of measures as the disintegration of the measure $\mu$ along the plaques of the chart $\phi$. Given this notion of disintegration, we make the following slight modification of the definition of a Gibbs expanding state found in \autocite[Def. 2.2]{saghin2018lyapunov}.
As before, we write $C^{1+}$ for an object that is $C^{1+\theta}$ for some $\theta>0$. Note that the notation $C^{1+}$ does not make an implicit statement of uniformity: different maps that are $C^{1+}$ may have different H\"older exponents.
\begin{defn}
 Let $\mc{F}$ be an expanding foliation for a $C^{1+}$ diffeomorphism $f$. An $f$-invariant measure $\mu$ is a \emph{Gibbs expanding} state along $\mc{F}$ if, for any foliation chart of $\mc{F}$, the disintegration of $\mu$ along the plaques of the chart is equivalent to the Lebesgue measure on the plaque for $\mu$-almost every plaque. 
\end{defn}

The following is an abridgement of a more general result, \autocite[Thm. 6]{saghin2018lyapunov}, to the one-dimensional case. However, that result has an additional hypothesis that the dynamics be $C^2$ instead of $C^{1+}$. We state the sharpened version of this result and outline the proof, which is essentially contained in the implication (B1) implies (B5') in \autocite{saghin2018lyapunov}.

\begin{lem}\label{lem:regularitycriterion}
Let $M$ be a smooth closed manifold, and let $f,g\in \Diff^{1+}(M)$. Let $\mc{F}$ is a one-dimensional expanding foliation for $f$, and let $\mc{G}$ be an expanding foliation for $g$ such that $\mc{F}$ and $\mc{G}$ have uniformly $C^{1+}$ leaves. Let $\mu$ be a Gibbs expanding state of $f$ along $\mc{F}$. Suppose that $f$ and $g$ are topologically conjugate by a homeomorphism $h$ and that $h$ intertwines $\mc{F}$ and $\mc{G}$. Then the following two conditions are equivalent:

\begin{enumerate}
\item
$\nu\coloneqq h_*(\mu)$ is a Gibbs expanding state of $g$ along the foliation $\mc{G}$.
\item
$h$ restricted to each $\mc{F}$ leaf within the support of $\mu$ is uniformly $C^{1+}$.
\end{enumerate}
\end{lem}

\begin{proof}
Fix a foliation box $B^f$ for $\mc{F}$. By this, we mean that $B^f$ is the image of a foliation chart $\psi\colon D_k\times D_{n-k}\to M$, where $D_k$ is a disk. As $h$ intertwines $\mc{F}$ and $\mc{G}$, the set $h(B^f)$, which we call $B^g$, is a foliation box of $\mc{G}$. As $\mu$ and $\nu$ are expanding states, one may explicitly calculate their densities against volume on the leaves of $\mc{F}$. Write $B^f(x)$ for the plaque containing $x$. Specifically, for almost every $x\in B^f$ the disintegration of $\mu$ along the plaque containing $x$ has density $\rho_f$ against volume, where 
\begin{align*}
\rho^f(z)&=\frac{\Delta(x,z)}{\int \Delta(x,z)\,d\vol_{B^f(x)}},\text{ and }\\
\Delta(x,z)&=\lim_{n\to \infty} \frac{\|Df^n\mid_{T_z\mc{F}}\|}{\|Df^n\mid_{T_x\mc{F}\|}}.
\end{align*}
One may check that this density is uniformly H\"older along almost every plaque of $\mc{F}$ \autocite[Prop. 2.3]{saghin2018lyapunov}. If we write $\mu_x$ for the disintegration of $\mu$ along the plaque $B^f(x)$, then we may show that $h_*(\mu_x)=\nu_{h(x)}$ for $\mu$-a.e. $x\in B^f$; one may prove this by using the essential uniqueness of the disintegration. Continuity then forces $h_*(\mu_x)=\nu_{h(x)}$ for every $x\in B^f$. Now consider the restriction of $h$ to a single plaque; we write $h_x\colon B^f(x)\to B^g(h(x))$ for the restricted map. We have already established that $h_*(\mu_x)=\nu_{h(x)}$. We also know that 
\[
\mu_x=\rho^f\,d\vol_{B^f(x)}\text{ and } \nu_{h(x)}=\rho^g\,d\vol_{B^g(h(x))},
\]
where $\rho^f$ and $\rho^g$ are uniformly H\"older. This implies that if $y\in B^f(x)$ then
\[
\int_x^y \rho^f\,d\vol_{B_x^f}=\int_{h(x)}^{h(y)}\rho^g\,d\vol_{B_{h(x)}^g}.
\]
The implicit function theorem then implies that $h_x$ is uniformly $C^{1+}$. By fixing a covering of $M$ by finitely many foliation boxes, we obtain a uniform estimate over all of $M$.
\end{proof}

\begin{rem}
We remark that Lemma \ref{lem:regularitycriterion} is one of the major obstacles to lowering the regularity from $C^{1+}$ in  Theorem \ref{thm:rigidity} to $C^1$. If the dynamics are only assumed to be $C^1$, then there is no a priori reason why the function $\Delta(x,z)$ that appears in the above proof would even be defined.
\end{rem}

We now introduce one final result that will be of use. Suppose that $\mc{F}$ is an expanding foliation for a $C^{1+}$ diffeomorphism $f$ of a manifold $M$ and that $\mu$ is an $f$-invariant measure.
Suppose that $\xi$ is a measurable partition of $M$. We say that $\xi$ is \emph{subordinate to $\mc{F}$ and }$\mu$ if 
\begin{enumerate}
\item
the common refinement $\vee_{i\le 0} f^i\xi$ is the partition into points;
\item
for all $x\in M$,  $\xi(x)$ is contained in a single $\mc{F}$ leaf;
\item
for $\mu$-a.e.\@ $x$, $\xi(x)$ is bounded and contains a neighborhood of $x$ in $\mc{F}(x)$.
\end{enumerate}

In addition, if the partition $f\xi$ is coarser than $\xi$ we say that $\xi$ is an \emph{increasing} partition.
The construction of increasing measurable partitions for expanding foliations is classical. In the $C^{1+}$ setting, see, for instance, \autocite[Sec. 3]{yang2016entropy}.

We now recall a useful result concerning the Pesin entropy formula. See \autocite[Sec. 2.4]{saghin2018lyapunov} for an explanation of this result in the present context. The argument there is an adaptation of the argument of Ledrappier presented in \autocite{ledrappier1985metric}.
\begin{lem}\label{lem:pesin}
Let $f$ be a $C^{1+}$ diffeomorphism and let $\mu$ be an $f$-invariant measure. Suppose that $\mc{F}$ is a $C^{1+}$ expanding foliation for $f$. Suppose that $\xi$ is an increasing measurable partition subordinate to $\mc{F}$ and $\mu$. Then the conditional measures of $\mu$ are absolutely continuous on the leaves of $\mc{F}$ if and only if
\[
H_{\mu}(f^{-1}\xi\mid \xi)=\int \log \|Df\mid_{\mc{TF}}\|\,d\mu,
\]
where $H_{\mu}(f^{-1}\xi\mid \xi)$ is the conditional entropy of $f^{-1}\xi$ given $\xi$.
\end{lem}

\section{Coarse geometry of Riemannian nilmanifolds and their automorphisms}\label{sec:coarse_geometry}

\subsection{Word metric on Lie groups}

In this section, we show that the right-invariant word metric on a Lie group and the Riemannian distance with respect to a right-invariant metric are quasi-isometric. The important idea used in this section, namely the function $\phi$ defined below, is due to Guivarc'h and was defined in \autocite{guivarch1973croissance}. For a recent use of these same estimates in a different context, see \autocite{cornulier2016gradings}.

Let $G$ be a connected Lie group. Fix a compact symmetric neighborhood $U$ of the identity of $G$ and a right-invariant metric $d_G$ on $G$. For $x,y\in G$, we define $d_U(x,y)$ to be the minimum $n$ such that $yx^{-1}=u_1u_2\cdots u_n$ where $u_i\in U$. Note that $d_U$ is right-invariant but not necessarily left-invariant. 
The following proposition is a special case of \autocite[Prop. 4.4]{breuillard2007geometry}.

\begin{prop}\label{prop:wordandriemannianmetricscoincide}
The metrics $d_U$ and $d_G$ are quasi-isometric: there exist $A\ge 1$,$B> 0$ such that for all $x,y\in G$
\[
\frac{1}{A}d_U(x,y)-B\le d_{G}(x,y)\le Ad_U(x,y)+B.
\]
\end{prop}

Suppose that $N$ is a nilpotent Lie group and write $\mf{n}$ for the Lie algebra of $N$. As before, write $\lcc{\mf{n}}{k}$ for the $k$th term in the lower central series of $\mf{n}$. A norm $\|\cdot \|$ on $\mf{n}$ induces a norm on $\lcc{\mf{n}}{k}/\lcc{\mf{n}}{k+1}$. Choose vector space complements $\cgrad{\mf{n}}{k}$ to $\lcc{\mf{n}}{k+1}$ inside of $\lcc{\mf{n}}{k}$. The norm restricts to a norm on these subspaces. Decompose an element $x\in \mf{n}$ as $\sum_k x_k$ where each $x_k\in \cgrad{\mf{n}}{k}$. Define the \emph{Guivarc'h length} of an element $x\in \mf{n}$ by
\[
\phi(x)=\max_k \|x_k\|^{1/k}.
\]
Note that if $\mf{n}$ is not abelian then $\phi$ is not a norm. 
The following theorem is implicit in the work of Guivarc'h \autocite{guivarch1973croissance}, though it does not seem to be explicitly stated. A thorough explication of Guivarc'h's  result is given in \autocite[Thm. 2.7]{breuillard2007geometry}.

\begin{thm}\label{thm:guivarchlength}
Let $N$ be a nilpotent Lie group endowed with a right-invariant Riemannian metric. Then there exist constants $A>0,B\ge 0$ such that for any $x\in N$, 
\[
\frac{1}{A}\phi(\log x)-B\le d_N(e,x)\le A\phi(\log x)+B.
\]
\end{thm}

Using this coarse estimate, we now return to the escape speeds defined in equation \eqref{defn:escapespeed}. The following proposition shows that points lying in the slowest subgroup of an expanding automorphism with simple real spectrum are characterized by their escape speed.

\begin{prop}\label{prop:expandingslowest}
Suppose that $N$ is a nilpotent Lie group and $L\colon N\to N$ is an expanding automorphism of $N$ with simple sorted spectrum. Let $\lambda$ be the smallest modulus eigenvalue of $L$ and let $v$ be an eigenvector of eigenvalue $\lambda$. Let $\Sigma$ be an eigenbasis for $L$ containing $v$. Let $\sigma=\min_{w\in \Sigma\setminus\{v\}} \sigma_w$, where $\sigma_w$ is the escape rate in direction $w$ as defined in equation \eqref{defn:escapespeed}. Then $\sigma>\sigma_{v}=\abs{\lambda}$. Choose any $\eta$ such that $\sigma_v<\eta<\sigma$. Then for any $x\in N$, if there exist $C,D$ such that $d_N(L^n(x),e)\le C\eta^n+D$ for all $n\ge 0$, then $x$ lies in the subgroup tangent to $v$.
\end{prop}

\begin{proof}
That $\sigma>\sigma_v$ is the content of Corollary \ref{cor:smallest_escape_rate}. Now suppose that $\sigma_v<\eta<\sigma$. Suppose that $x\in N$ and that there exist $C, D$ such that $d_N(L^n(x),e)<C\eta^n+D$ for all $n\ge 0$. Write $\log(x)=\sum_k x_k$. Where $x_k\in \mf{n}_{(k)}$. Then $L$ acts on $\log(x)$ by scaling each of its components in the eigenspace decomposition. Specifically, write $\log(x)=\sum_{1\le i\le r}\sum_{1\le j \le \dim \mf{n}_{(i)}} a_{ij}v_{i,j}$ where $v_{i,j}$ is is the $j$th slowest unstable eigenvector in $\mf{n}_{(i)}$. Consequently, $dL^n(\log(x))=\sum_{i}\sum_j \lambda_{v_{i,j}}^n a_{ij}v_{i,j}$. Now, $dL^n(\log x)=\log L^n(x)$. By Theorem \ref{thm:guivarchlength}, there exist $A$ and $B$ such that $Ad_N(x,e)+B\ge \phi(\log (x))$. Thus, for sufficiently large $n$,
\begin{align*}
AC\eta^n+(B+AD)&\ge \phi(\log(L^n))\\
&=\phi\pez{\sum_i \sum_j \lambda_{v_{i,j}}^n a_{ij}v_{i,j}}\\
&=\max_i \|\sum_j \lambda_{v_{i,j}}^n a_{ij}v_{i,j}\|^{1/i}\\
&\ge (\sigma_{v_{i,j}})^n\abs{a_{ij}}^{1/i}\|v_{i,j}\|^{1/i}\\
\end{align*}
But by choice $\eta<\sigma_{v_{i,j}}$ for all $v_{i,j}$ except $v_{1,1}$. Thus $a_{ij}=0$ except for $i=1,j=1$. Thus $x$ lies in the subgroup tangent to $v_{1,1}$ as its logarithm is a multiple of $v_{1,1}$.
\end{proof}

Note that the proof of Proposition \ref{prop:expandingslowest} provides detailed information about the distance between $e$ and $L^n(x)$. However, we have only extracted the above statement concerning the slowest speed because it is all we need.

\begin{thm}\label{thm:V_i_characterization}
Suppose that $N/\Gamma$ is a nilmanifold and $L$ is an automorphism of $N/\Gamma$ with simple sorted spectrum. Fix some $i$;
then by Proposition \ref{prop:dynamical_foliations_for_nilmanifold_automorphisms}, the $i$th strong foliation $\mc{S}^{u,L}_i$ exists. By the same proposition, the $i$th weak foliation $\mc{W}_i^{u,L}$ exists and subfoliates $\mc{S}^{u,L}_i$.
Let $\Sigma$ be an eigenbasis for the action of $dL$ on $\mf{s}_i^u$ and let $v$ be the vector with smallest modulus eigenvalue in $\Sigma$. Let $\sigma=\min_{w\in \Sigma\setminus{\{v\}}} \sigma_w$ be the second slowest escape speed of the action of $dL$ on $\mf{s}_i^u$. Then $\sigma>\sigma_v$, where $\sigma_v$ denotes the escape speed of $v$ associated to the action of $dL$ on $\mf{s}_i^u$. Choose $\eta$ such that $\sigma_v<\eta<\sigma$. Then for any $x\in N/\Gamma$ and $y\in \mc{S}^{u,L}_i(x)$, if there exist $C,D$ such that $d_{\mc{S}^{u,L}_i}(L^n(x),L^n(y))\le C\eta^n+D$ for all $n\ge 0$, then $x\in \mc{W}^{u,L}_i(y)$.
\end{thm}

\begin{proof}
We reduce to Proposition \ref{prop:expandingslowest}. Suppose that that for some $C,D$, $d_{\mc{S}^{u,L}_i}(L^n(x),L^n(y))<C\eta^n+D$ for all $n\ge 0$. As distance along unstable leaves is the same in $N/\Gamma$ or in the lifted foliation on $N$, it suffices to work in the universal cover. In the universal cover, the lifted foliation has leaf through $x$ equal to $S^u_ix$, where $S^u_i$ is the strong subgroup defined in Proposition \ref{prop:dynamical_foliations_for_nilmanifold_automorphisms}. Consequently, we may write $y=mx$ for some $m\in S^u_i$. By assumption, $d_{\mc{S}^{u,L}_i}(L^n(x),L^n(y))\le C\eta^n+D$ for all $n\ge 0$. The strong unstable foliation is preserved by right multiplication. Right multiplication preserves the distance along leaves as the leaf metric is induced by a right-invariant metric on $N$. Thus,
\[
d_{\mc{S}^{u,L}_i}(L^n(x),L^n(y))=d_{\mc{S}^{u,L}_i}(e,L^n(y)(L^n(x))^{-1})=d_{\mc{S}^{u,L}_i}(e,L^n(yx^{-1}))=d_{\mc{S}^{u,L}_i}(e,L^n(m))
\]

The restriction of $L$ to $S^{u}_i$ is an expanding automorphism with simple sorted spectrum, and so the previous proposition applies with the same choice of $\eta$. We conclude that $m$ lies in the subgroup of $S^{u}_i$ generated by $v$, as desired.
\end{proof}

\section{Coarse geometry and conjugacies}

In the first subsection, we show a useful result that gives that a conjugacy intertwining the leaves of two sufficiently nice foliations induces a quasi-isometry between the leaves of those foliations. In the second section, we use this result to show that under suitable conditions the $E_i^{u,f}$ distribution is uniquely integrable. The proof of unique integrability is obtained by using that a quasi-isometry respects escape speeds.

Before we begin, we record a basic result showing that a conjugacy interwines stable manifolds. Recall that $\mc{S}_1^{u,f}$ is equal to the full unstable foliation.
\begin{prop}\label{prop:unstabletounstable}
Suppose that $L$ is an Anosov automorphism, $f$ is an Anosov diffeomorphism, and $h$ is a conjugacy satisfying $h\circ f=L\circ h$. Then $h(\mc{S}^{u,f}_1)=\mc{S}_1^{u,L}$ and $h(\mc{S}^{s,f}_1)=\mc{S}^{s,L}_1$.
\end{prop}

\begin{proof}
Two points are in the same stable manifold if and only if they converge to each other under forward interation. If $d(f^n(x),f^n(y))\to 0$, then $d(h(f^n(x)),h(f^n(y)))=d(L^n(h(x)),L^n(h(y)))\to 0$ as $h$ is continuous. So, if $x\in S^{s,f}(y)_1$, then $h(x)\in S_1^{s,L}(h(y))$. The proof is similar in the case of unstable manifolds.
\end{proof}

\subsection{Bounded geometry}

When working with non-compact Riemannian manifolds, it is easy to accidentally allow trivial counterexamples to seemingly reasonable analytic claims. However, there exists a natural class of manifolds that are suitable for analysis: those with bounded geometry.
For further discussion and examples, see \autocite[Ch. 2]{eldering2013normally}, which gives an extended discussion of bounded geometry in a dynamical setting.

A Riemannian metric allows one to judge the size of the jets of a function. For instance, given a Riemannian metric the differential of a map at a point may be viewed as a linear map between two normed spaces. Consequently, the differential at a point is normed by the operator norm.
Similarly, a choice of metric induces a norm on all other bundles and maps that one might consider. Whenever we talk about $C^k$ norms on a Riemannian manifold, we are referring to these uniform norms defined with reference to a metric. One could also define this topology by selecting a distinguished family of charts on the manifold.  We refer to these $C^k$ norms as the uniform norm on a Riemannian manifold.

This approach to the $C^k$ norms may seem displeasing because it is non-canonical. The space $C^{k}$ may no longer contain all $C^k$ maps: some may have infinite norm. Further, if one has a Riemmian manifold with unbounded curvature, then the curvature of this manifold has infinite norm, and thus is not even an element of the space of $C^1$ tensors over the manifold in the uniform norm.  Consequently, we will distinguish the spaces $C^k$ and $C^k_u$. The former is the usual space of maps with $k$ continuous derivatives, and the latter is the subset of $C^k$ where there is a uniform estimate on the derivatives.

Dynamicists often prefer the convenience of working directly in charts. Consequently, when making arguments it is very useful if one has a uniform lower bound on the size of an exponential chart on a Riemannian manifold.  These considerations all bring us to a class of manifolds whose uniform $C^k$ topology is suitable for analysis.

\begin{defn}
We say that a smooth Riemannian manifold $N$ equipped with a smooth metric, $g$, has \emph{bounded geometry} if 
\begin{enumerate}
\item
The global injectivity radius of $N$ is positive; i.e. there is a uniform lower bound on the injectivity radius of the exponential map over all points $n\in N$.
\item
For each $k\ge 0$, there exists $C_k$ such that pointwise
\[
\|\nabla^k R \|\le C_k,
\]
where $\|\cdot\|$ is the norm on tensors induced by $g$ and $R$ denotes the curvature of the Levi-Civita connection $\nabla$ of $g$.
\end{enumerate}
\end{defn}

Consider the universal cover $\wt{N}$ of a compact Riemannian manifold $N$. Endowed with the pullback metric from $N$, $\wt{N}$ has bounded geometry. For an example of a manifold without bounded geometry, consider a Riemannian manifold with a two-dimensional cusp: for any $\epsilon>0$, there is a point $x$ sufficiently deep in the cusp that the injectivity radius of the exponential map at $x$ is less than $\epsilon$.

We also introduce a class of submanifolds of bounded geometry manifolds that are also suitable for analysis.

\begin{defn}\autocite[Def. 2.21]{eldering2013normally}
Let $k\ge 1$ be an integer. Let $N$ be a Riemannian manifold of bounded geometry and let $i\colon M\to N$ be a $C^k$ immersion of a $C^k$ manifold $M$ into $N$. For $x\in M$, we write $M_{x,\delta}$ for the connected component of $x$ in  $i^{-1}(B_{\delta}(i(x))\cap i(M))$, where $B_{\delta}(i(x))$ is the open ball of radius $\delta$ in $N$ centered at $i(x)$. We say that $M$ is a $C^k$ \emph{uniformly immersed submanifold of $N$} when there exists $\delta>0$ such that for all $x\in M$, $M_{x,\delta}$  is represented in normal coordinates on $N$ by the graph of a function $g_x\colon T_xM\to T_xM^{\perp}$. We also require that there is a uniform bound over all $x$ on the $C^k$ norm of $g_x$, where this norm is defined with respect to the natural Euclidean structure on $T_xM$ and $T_xM^{\perp}$.
\end{defn}

Note that our definition of uniformly immersed does not reference the pullback metric on $M$.
The reason for this is that
the definition of bounded geometry includes reference to the exponential map, which need not be defined if $M$ or the pullback metric has regularity lower than $C^2$. However, in the case that $i$ and $M$ are both sufficiently regular being uniformly immersed is equivalent to $M$ having low order bounded geometry \autocite[Lem. 2.27]{eldering2013normally}.
One should also note that the definition is quite adapted to a dynamical setting. In the graph transform approach to constructing unstable manifolds, one essentially obtains the needed estimate in the course of constructing the unstable manifolds.

\begin{prop}\label{prop:strong_is_uniformly_immersed}
Suppose that $f$ is a $C^{1+}$ partially hyperbolic map of a smooth compact manifold $M$ and that $\mc{S}^{u,f}$ is the strong unstable foliation of $M$ defined by the partially hyperbolic splitting of $f$. Then $\mc{S}^{u,f}$ is a foliation with uniformly $C^{1+}$ leaves. Moreover, each leaf of $\mc{S}^{u,f}$ is a $C^1$ uniformly immersed submanifold of $M$.
\end{prop}

\begin{proof}
This follows because the strong foliation $\mc{S}^{u,f}$ admits a $C^1$ plaquation, see \autocite[Cor. 5.6]{hirsch1977invariant}. 
Let $\text{Emb}^1(D^k,M)$ be the space of $C^1$ embeddings of the closed $k$-dimensional unit disk in $M$ endowed with the $C^1$ topology.
Admitting such a plaquation implies that there is a subset $\mc{U}\subseteq \text{Emb}^1(D^k, M)$, where $k$ is the dimension of a leaf of $\mc{S}^{u,f}$,  such that each point in the foliation is in the image of such a disk. The precompactness of $\mc{U}$ immediately implies the uniformity estimate in the definition of uniformly immersed.
\end{proof}

The following proposition shows that for a uniformly immersed submanifold we may locally approximate distance along a leaf by the distance in the manifold.

\begin{prop}\label{prop:bounded_approximation}
\autocite[Lem. 2.25]{eldering2013normally}
Let $M$ be a $C^1$ uniformly immersed submanifold of a smooth Riemannian manifold $N$ that has bounded geometry. Write $d_M$ for the Riemannian distance on $M$ induced by the pullback metric and $d_N$ for the Riemannian distance on $N$. Then for any $C>1$ there exists $\delta$ such that if $d_M(x,y)<\delta$, then
\[
d_N(x,y)\le d_M(x,y)\le Cd_N(x,y).
\]
\end{prop}

For a full proof see \autocite[Lem. 2.25]{eldering2013normally}. However, we will describe briefly the idea. The assumption of bounded geometry allows us to locally approximate the geometry $M$ by the geometry of a graph $g\colon B_{\delta}(0)\subseteq \R^m\to \R^k$ where $m$ is the dimension of $M$, $k$ is the codimension of $M$ in $N$, and $D_0g=0$. Suppose $(p,g(p))$ and $(q,g(q))$ are two points on the graph of $g$. If $\delta$ is chosen to be sufficiently small, then the distance between $(p,g(p))$ and $(q,g(q))$ is approximately the same as the distance between $p$ and $q$ in $\R^m$. More needs to be said, particularly about uniformity, but this is essentially the idea.

\begin{lem}
Suppose that $h\colon N\to N$ is a continuous map of a compact smooth Riemannian manifold $N$. Suppose that $M$ and $M'$ are two $C^1$ uniformly injectively immersed submanifolds of $M$ and that $h\colon M\to M'$ is a bijection. Then $h\mid_{M}$ is a quasi-isometry from $M$ to $M'$ in the induced metric.
\end{lem}

\begin{proof}
Let $\eta$ and $C$ be constants such that the conclusion of Proposition \ref{prop:bounded_approximation} holds for $M$ and $M'$. Because
$h$ is a map of a compact manifold there exists $\delta$, such that if $d_{N}(x,y)<\delta$, then $d_N(h(x),h(y))<\eta/2$. Fix a minimum length path $\gamma$ between two points $x$ and $y$ in $M$ such that $d_M(x,y)>\delta/C$.  Divide $\gamma$ into $n$ segments of length $\delta/C$ and one segment, the last, of length less than $\delta/C$. Write the endpoints of these segments as $x_1,\ldots, x_{n}=y$. We then have that 
\[
d_{M'}(h(x),h(y))\le \pez{\sum_{i=1}^{n-2} d_{M'}(h(x_i),h(x_{i+1}))}+d_{M'}(h(x_{n-1}),h(x_{n})).
\]
As $d_M(x_i,x_{i+1})<\delta/C$, we see by the bilipschitz estimate from Proposition \ref{prop:bounded_approximation} that $d_N(x_i,x_{i+1})<\delta$, and hence $d_N(h(x_i),h(x_{i+1}))<\eta/2$. By the bilipschitz estimate again, this time on $M'$, $d_{M'}(h(x_i),h(x_{i+1})\le C\eta/2 $. So, we see that 
\[
d_{M'}(h(x),h(y))\le n\cdot C\eta/2, 
\]
but $n=\lceil d_M(x,y)/(\delta/C)\rceil$, so
\[
d_{M'}(h(x),h(y))\le d_M(x,y)\frac{\eta}{\delta}C^2+1.
\]

To obtain the lower bound, do the same argument using $h^{-1}$. This gives that there exist constants $C',D'$ such that 
\[
d_{M}(h^{-1}(x),h^{-1}(y)) \le C'd_{M'}(x,y)+D
\]
Thus by rearranging quasi-isometry follows.
\end{proof}

Suppose that $\mc{F}$ is a foliation with uniformly $C^1$ leaves and let $\mc{F}'$ be the space that is topologized and given the smooth structure of the disjoint union of the leaves of $\mc{F}$. The inclusion of $\mc{F}'$ into $M$ is a $C^1$ uniform immersion, so the conclusion of Proposition
\ref{prop:bounded_approximation} holds with a uniform constant over the inclusion of all leaves of $\mc{F}$ into $M$. Applying this observation to a map $h$ intertwining two such foliations yields the following corollary.
\begin{cor}\label{cor:QI_foliations}
Suppose that $\mc{F}$ and $\mc{G}$ are two topological foliations with uniformly $C^1$ leaves of a smooth compact manifold $M$. If $h\colon M\to M$ is homeomorphism that intertwines the $\mc{F}$ and $\mc{G}$ foliations, then for all $x\in M$, $h$ is a quasi-isometry from $\mc{F}(x)$ to $\mc{G}(h(x))$ and the constants of the quasi-isometry can be taken to be uniform over all leaves of $\mc{F}$.
\end{cor}

\subsection{Quasi-isometry and unique integrability}

Using Corollary \ref{cor:QI_foliations}, we show in this subsection that, under certain hypotheses, the $E_i^{u,f}$ distribution uniquely integrates to a $1$-dimensional foliation.

\begin{prop}\label{prop:slowtoslow}
Suppose that $L$ is a Anosov automorphim of a nilmanifold $N/\Gamma$ with simple sorted spectrum. Then there exists a $C^1$ neighborhood $\mc{U}$ of $L$ in $\Diff^{1+}(M)$ with the following property. If $f\in \mc{U}$ and $h_f$ is a conjugacy between $f$ and $L$, and, if, for some $i$, $h_f(\mc{S}_i^{u,f})=\mc{S}_i^{u,L}$, then the $E_i^{u,f}$ distribution is uniquely integrable and integrates to a foliation $\mc{W}_i^{u,f}$ with uniformly $C^{1+}$ leaves. In addition,
\begin{enumerate}
\item
For each $x\in N/\Gamma$, $h_f(\mc{W}_i^{u,f}(x))=\mc{W}_i^{u,L}(h_f(x))$.
\item
The $\mc{W}_i^{u,f}$ foliation and the $\mc{S}_{1+1}^{u,f}$ foliation have subordinate product structure to the $\mc{S}^{u,f}_i$ foliation.
\item
The $\mc{S}^{u,f}_{i+1}$ holonomies between leaves of $\mc{W}_i^{u,f}$ are uniformly $C^{1+}$.
\end{enumerate}

\end{prop}

\begin{proof}

It suffices to construct such a neighborhood for each $i$; intersecting these neighborhoods then gives the result. 

Fix some $1\le i\le \dim E^u$.
We apply Theorem \ref{thm:V_i_characterization} on the manifold $\mc{S}^{u,L}_i$. 
Let $\Sigma$ be an eigenbasis for the action of $L$ on $\mf{s}^u_i$ that contains $v$, a vector tangent to the subgroup $W_i^u$. Let $\sigma=\min_{w\in \Sigma\setminus{\{v\}}} \sigma_w>\abs{\lambda_i^u}$. Note that the $\sigma_w$ are the escape rates of the action on $S_i^u$. Choose $\eta$ such that $\sigma_v<\eta<\sigma$. Then for any $x\in N/\Gamma$ and $y\in \mc{S}^{u,L}_i(x)$, if there exist $C,D$ such that if $d_{\mc{S}^{u,L}_i}(L^n(x),L^n(y))\le C\eta^n+D$ for all $n\ge 0$, then, by Theorem \ref{thm:V_i_characterization}, $x$ lies in the same $\mc{S}_i^{u,L}$ leaf as $y$. 

Pick $\epsilon>0$ such that $\abs{\lambda_i^u}+\epsilon<\sigma$. Then by Theorem \ref{thm:mather_spectrum_perturbation}, there exists a $C^1$ neighborhood $\mc{U}_{\epsilon}$ of $L$ such that for any unit vector $w\in E_{i}^{u,f}$, we have $\|Dfw\|\le (\abs{\lambda_i^u}+\epsilon)<\eta$.

Suppose now that $f\in \mc{U}_\epsilon$ and that $h$ is a conjugacy between $f$ and $L$ intertwining the $i$th strong foliations. Then all of the previously mentioned considerations hold for $f$. Suppose that $x\in N/\Gamma$ and let $\gamma$ be a curve tangent to the $E_i^{u,f}$ distribution containing $x$. Suppose that $y$ is another point on $\gamma$.
Write $\ell(\gamma)$ for the length of $\gamma$. The inequality 
$\|Dfw\|\le (\abs{\lambda_i^u}+\epsilon)<\sigma$ implies
\[
\ell(f^n\circ \gamma)\le (\abs{\lambda_i^u}+\epsilon)^n\ell(\gamma),
\]
and hence,
\[
d_{\mc{S}_i^{u,f}}(f^n(x),f^n(y))\le (\abs{\lambda_i^u}+\epsilon)^n\ell(\gamma).
\]

By Proposition \ref{prop:strong_is_uniformly_immersed} the $\mc{S}_i^{u,f}$ foliation has uniformly $C^1$ leaves and so by Corollary \ref{cor:QI_foliations} there exist constants $C,D$ such that for all $n\ge 0$
\begin{equation}
d_{\mc{S}^{u,L}_i}(h(f^n(x)),h(f^n(y)))\le Cd_{\mc{S}^{u,f}_i}(f^n(x),f^n(y))+D.
\end{equation}
But as $h\circ f^n=L^n\circ h$, this implies that 
\[
d_{\mc{S}^{u,L}_i}(L^n(h(x)),L^n(h(y)))\le C(\abs{\lambda_i^u}+\epsilon)^n+D.
\]
Consequently, as $\abs{\lambda_i^u}+\epsilon<\sigma$, we see that $h(x)\in \mc{W}_i^{u,L}(h(y))$. Thus $h(\gamma)\subseteq \mc{W}_i^{u,L}(h(x))$. This implies that $E_i^{u,f}$ is uniquely integrable and integrates to the $h^{-1}(\mc{W}_i^{u,L})$ foliation. 

Finally, we need to show the claim about subordinate product structure. The dimensions of the foliations are correct, so we just need to show that a $\mc{W}_i^{u,f}$ leaf and a $\mc{S}_{i+1}^{u,f}$ leaf intersect at exactly one point if they lie in the same $\mc{S}_i^{u,f}$ leaf.

To see that there is at most one point of intersection, note that as $\mc{S}_{i+1}^{u,f}$ and $\mc{W}_i^{u,f}$ are uniformly transverse there is a uniform lower bound on the distance between points of $\mc{S}_{i+1}^{u,f}(x)\cap \mc{W}_i^{u,f}(x)$ independent of $x$. If there were another point $y\in \mc{S}_{i+1}^{u,f}(x)\cap \mc{W}_i^{u,f}(x)$, then by iterating the dynamics backwards, we would obtain points $f^{-n}(x)$ and $f^{-n}(y)$ arbitrarily close to each other and with $f^{-n}(y)\in \mc{S}_{i+1}^{u,f}(x)\cap \mc{W}_i^{u,f}(x)$. This contradicts our previous observation about transversality.

Next, we show how to deduce that there exists a point of intersection. The argument is similar: the distributions that the $\mc{W}_i^{u,f}$ and $\mc{S}_{i+1}^{u,f}$ foliations are tangent to are uniformly transverse. Consequently, there exists $\epsilon>0$ such that if there are two points $x,y\in \mc{S}_i^{u,f}(z)$ and $d_{\mc{S}_i^{u,f}}(x,y)<\epsilon$, then $\mc{W}_i^{u,f}(x)\cap \mc{S}_{i+1}^{u,f}(y)$ is non-empty. So, to see that for $y\in \mc{S}_{i}^{u,f}(x)$ that $\mc{W}_i^{u,f}(x)\cap \mc{S}_{i+1}^{u,f}(x)$ intersect, observe that for sufficiently large $n$, we have $d_{\mc{S}_i^{u,f}}(f^{-n}(x),f^{-n}(y))<\epsilon$. This concludes the first two claims.

For the claim about the holonomies, see \autocite[Sec. 2.2]{brown2016smoothness}, which gives this result in the $C^{1+}$ setting. Also, see the appendix of \autocite{barreira1999dimension}.
\end{proof}

\section{Weak and strong distance along foliations}

Suppose that two foliations $\mc{F}$ and $\mc{G}$ have subordinate product structure to a foliation $\mc{W}$ of a Riemannian manifold $M$ and that all three foliations have uniformly $C^1$ leaves. Fix two points $x,y\in M$ such that $x\in \mc{W}(y)$. Then one can consider the distance $d_{\mc{F}}(x,\mc{F}(x)\cap \mc{G}(y))$. While one might not expect such a notion of distance to be useful in general, we will study such a notion in an algebraic setting that permits a substantial application.
As before, fix a nilmanifold $N/\Gamma$ and an Anosov automorphism $L$ that is sorted and irreducible. In this section, we study this construction in the case of the foliations $\mc{W}_i^{u,L}$ and $\mc{S}_{i+1}^{u,L}$ subordinate to the foliation $\mc{S}_i^{u,L}$.

We begin by studying the analogous foliations on $N$.
As we have seen in Proposition \ref{prop:dynamical_foliations_for_nilmanifold_automorphisms}, the $\mc{S}_i^{u,L}$ foliation of $N$ is subfoliated by the $\mc{S}_{i+1}^{u,L}$ foliation as well as the $\mc{W}_i^{u,L}$ foliation.

\begin{defn}
Suppose that $L$ is an Anosov automorphism of the nilpotent group $N$ endowed with a right-invariant metric. For $1\le i\le \dim N^u$, we define the \emph{$i$th weak} and \emph{strong distances} on $\mc{S}_i^{u,L}(x)=S_i^ux$ as follows. For two points, $q,r\in \mc{S}_i^{u,L}(x)$, we define 
\begin{align*}
d_{\mc{W}_i}(q,r)&=d_{\mc{W}_i^{u,L}}(S_{i+1}^{u}q\cap W_i^{u}r,r),\\
d_{\mc{S}_i}(q,r)&=d_{\mc{S}_{i+1}^{u,L}}(q,S_{i+1}^uq\cap W_i^{u}r).\\
\end{align*}

\end{defn}

In other words, we find the point $z=\mc{S}_{i+1}^{u,L}(q)\cap \mc{W}_i^{u,L}(r)$ and measure the distance from $q$ to $z$ along the $(i+1)$st strong unstable foliation to define the strong distance.  One measures the distance from $z$ to $r$ along the $\mc{W}_i^{u,L}$ foliation to find the weak distance.  Although we have used the term ``distance" here, note that this notion is not symmetric. 
Further, note that this notion depends on the automorphism $L$ used to define the strong and weak foliations. We do not include the automorphism in the notation as we only ever consider one automorphism at a time.

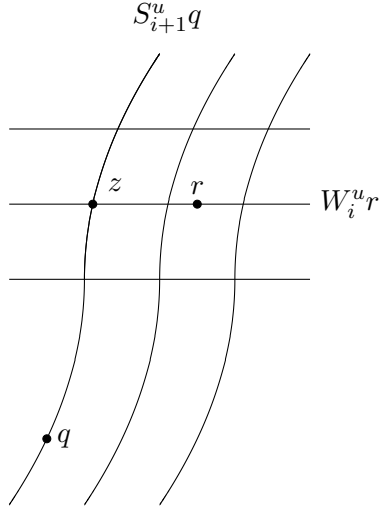
\begin{figure}
\centering
\begin{tikzpicture}[domain=0:4, samples=4000]
    \draw[color=black]    plot[domain=1:2] ({\x},{3*(\x-1)^(1/2)})             node[label={[xshift=0.1cm, yshift=0.0cm]$S_{i+1}^uq$}] {}; 
    \draw[name path=phi, color=black]    plot[domain=0:1] ({\x},{-3*(1-\x)^(1/2)})           ; 
    \draw[color=black]    plot[domain=0:4] ({\x},{1})             node[right] {$W_i^ur$}; 
    \draw[color=black]    plot[domain=0:4] ({\x},{2}); 
    \draw[color=black]    plot[domain=0:4] ({\x},{0}); 
    \filldraw (1.111111,1) circle (1.5pt) node[label={[xshift=0.3cm, yshift=-0.1cm]$z$}] {}; 
    \filldraw (2.5,1) circle (1.5pt) node[above] {$r$};
    \filldraw (.5,{-3*(1-.5)^(1/2)}) circle (1.5pt) node[right] {$q$};
    \draw[color=black]    plot[domain=1:2] ({\x},{3*(\x-1)^(1/2)})             ; 
    \draw[color=black]    plot[domain=1:2] ({\x},{-3*(2-\x)^(1/2)})             ; 
    \draw[color=black]    plot[domain=2:3] ({\x},{3*(\x-2)^(1/2)})    ; 
    \draw[color=black]    plot[domain=2:3] ({\x},{-3*(3-\x)^(1/2)})             ; 
    \draw[color=black]    plot[domain=3:4] ({\x},{3*(\x-3)^(1/2)})    ; 
\end{tikzpicture}
\caption{A diagrammatic depiction of the weak and strong foliations illustrating the points involved in the definition of the weak and strong distances. The strong distance between $q$ and $r$ is the distance from $z$ to $q$ along the $\mc{S}_{i+1}^u$ foliation. The weak distance between $q$ and $r$ is the distance between $z$ and $r$ along the $\mc{W}_i^u$ foliation.}
\end{figure}

The following is 
immediate due to the right-invariance of the dynamical foliations and the right-invariance of distance measured along these foliations.
\begin{lem}
The weak and strong distances are invariant under right multiplication by elements of $N$, i.e. 
for $n\in N$, and any $q,r\in S^u_{i}x$,
\begin{align*}
d_{\mc{W}_i}(q,r)&=d_{\mc{W}_i}(qn,rn),\\
d_{\mc{S}_i}(q,r)&=d_{\mc{S}_i}(qn,rn).
\end{align*}
\end{lem}

We now record some lemmas concerning weak and strong distances that will be of use later.

\begin{lem}\label{lem:vh_distance_invariance}

Suppose that $w\in W_i^u$ and $s\in S^{u}_{i+1}$.  Then 
\begin{align*}
d_{\mc{W}_i}(e,ws)&=d_{\mc{W}_i}(e,w), \text{ and}\\
d_{\mc{S}_i}(e,ws)&=d_{\mc{S}_i}(e,s).
\end{align*}
\end{lem}

\begin{proof}
We use right-invariance of the distance. By definition,
\begin{align*}
d_{\mc{W}_i}(e,ws)&=d_{\mc{W}^{u,L}_i}(S^u_{i+1}e\cap W_i^u ws,ws)\\
&=d_{\mc{W}^{u,L}_i}(S^u_{i+1}\cap W_i^us,ws)\\
&=d_{\mc{W}^{u,L}_i}(s,ws)\\
&=d_{\mc{W}^{u,L}_i}(e,w)\\
&=d_{\mc{W}^{u,L}_i}(S^u_{i+1}\cap W_i^uw,w)\\
&=d_{\mc{W}_i}(e,w).
\end{align*}
For the second claim,
\begin{align*}
d_{\mc{S}_i}(e,ws)&=d_{\mc{S}_{i+1}^{u,L}} (e, S_{i+1}^u\cap W_i^u ws)\\
&=d_{\mc{S}_{i+1}^{u,L}}(e,S_{i+1}^u\cap W_i^us)\\
&=d_{\mc{S}_{i+1}}(e,s).
\end{align*}
\end{proof}

Before the next proof, we record a useful observation about one-parameter subgroups.

\begin{obs}\label{obs:one_parameter_subgroup_distance}
Suppose that $G$ is a Lie group endowed with a right-invariant metric. Fix an abelian subalgebra $\mf{h}\in \mf{g}$, and let $H\coloneqq \exp(\mf{h})$ be the analytic subgroup of $G$ tangent to $\mf{h}$ at $e\in G$. Then $\exp=\Exp$, where $\Exp$ is the Riemannian exponential map on $H$ in the metric given by the restriction of the metric on $G$.
\end{obs}

\begin{proof}
The right-invariant metric on $G$ restricted to $H$ is a bi-invariant metric on the abelian group $H$. The exponential map of any invariant metric on an abelian Lie group is the Lie exponential. Since $H$ is abelian $\Exp$ and $\exp$ must coincide.
\end{proof}

\begin{lem}\label{lem:distance_growth}
Let $L$ be an Anosov automorphism of a nilpotent Lie group $N$ with sorted simple spectrum. Suppose that $w\in W_i^u$ and $s\in S^u_{i+1}$. Then there exist $C,D>0$ such that for all $m\in \Z$,
\[
d_{\mc{W}_i}(e,w^ms^m)=mD,
\]
and
\[
d_{\mc{S}_i}(e,w^ms^m)\le mC.
\]
\end{lem}

\begin{proof}

For the first claim, as $s$ is in $S^{u}_{i+1}$, we see by the previous lemma that
\[
d_{\mc{W}_i}(e,w^m s^m)=d_{\mc{W}_i}(e,w^m).
\]
There exists $v\in \mf{n}$ such that $w=\exp(v)$, so $w^m=\exp(mv)$. Consequently,
\[
d_{\mc{W}_i}(e,w^m)=d_{\mc{W}_i}(e,\exp(mv))=d_{W_i^{u}}(e,\Exp(mv))=md_{W_i^{u}}(e,\exp(v)),
\]
as the Lie and Riemannian exponential coincide by the previous observation and $W_i^u$ is flat. Thus the equality follows.

We now show the second claim. By Lemma \ref{lem:vh_distance_invariance}, $d_{\mc{S}_i}(e,w^ms^m)=d_{\mc{S}_i}(e,s^m)$. Suppose that $\exp(v)=s$, so that $\exp(mv)=s^m$. The one-parameter subgroup $P=\exp(tv)$ is isomorphic to $\R$ as a Lie group. When endowed with the restriction of a right-invariant metric on $N$, the observation gives that $d_P(e,s^m)=md_P(e,s)$.  Note that $d_{\mc{S}_i}(e,s^m)\le d_P(e,s^m)$, as $P$ is endowed with the restricted metric and so $d_{\mc{S}_i}(e,s^m)\le md_P(e,s)$.
\end{proof}

\subsection{Weak and strong distance and automorphisms}

Here we study the change in weak and strong distance under an automorphism of $N/\Gamma$. As before, we consider an Anosov automorphism $L$ of $N/\Gamma$ such that $L$ has simple sorted spectrum.

\begin{claim}\label{lem:horizontal_distance_growth}
Suppose that $L$ is an Anosov automorphism of a nilpotent Lie group $N$ with simple sorted spectrum. 
Suppose that $x,y\in N$ with $y\in \mc{S}^{u,L}_{i}(x)$. Then for all $m\in \Z$,
\begin{equation}
d_{\mc{W}_i}(L^m(x),L^m(y))=\abs{\lambda^u_i}^md_{\mc{W}_i}(x,y).
\end{equation}

\end{claim}

\begin{proof}

Note that $d_{\mc{W}_i}(L^m(x),L^m(y))=d_{\mc{W}_i}(e,L^m(yx^{-1}))$ by right-invariance of horizontal distance. So, writing $yx^{-1}$ as $ws$ for some $w\in W_i^u$ and $s\in S_{i+1}^u$, we see that 
\[
d_{\mc{W}_i}(L^m(x),L^m(y))=d_{\mc{W}_i}(e,L^m(w)L^m(s))=d_{\mc{W}_i}(e,L^m(w)),
\]
by Lemma \ref{lem:vh_distance_invariance}.

Note that if $\exp(v)=w$, then $L^m(w)=\exp(dL_e^mv)=\exp((\lambda_i^u)^m v)$. By appealing to Observation \ref{obs:one_parameter_subgroup_distance}, we see immediately that $d_{W_i^{u}}(e,\exp(tv))=\abs{t}d_{W_i^u}(e,w)$. Thus,
\[
d_{\mc{W}_i}(e, L^m(w))=\abs{\lambda_i^u}^md_{\mc{W}_i^u}(e,w)=\abs{\lambda_i^u}^md_{\mc{W}_i}(e,w)=\abs{\lambda_i^u}^md_{\mc{W}_i}(e,ws)=\abs{\lambda_i^u}^md_{\mc{W}_i}(x,y).
\]
\end{proof}

\section{Irreducibility}\label{sec:irreducibility}

In this section, we discuss irreducibility and show the equivalence between its algebraic and dynamical formulations. 

\subsection{The toral case}

We begin by recalling the existing notion of irreducibility for toral automorphisms.
An element $A\in \SL(n,\Z)$ is said to be \emph{irreducible} if the characteristic polynomial of $A$  is irreducible over $\Q$. This notion appears in \autocite{gogolev2011local} and \autocite{gogolev2018local}. This is equivalent to each eigenspace of $A$ being dense when projected to $\mathbb{T}^n$.
If $A$ is not irreducible, it is shown in \autocite[Sec. 3]{gogolev2008smooth}, that $A$ is not locally Lyapunov spectrum rigid. There is a small oversight in the argument of \autocite{gogolev2008smooth}, which occurs in the case of non-simple spectrum. We will not explain this here, however, as the argument below, which is based on \autocite[Sec. 3]{gogolev2008smooth}, is self-contained.

\subsection{Nilmanifold case}

We now consider nilmanifolds. Fix an automorphism, $L$, of $N/\Gamma$. As in the preliminaries, we write $E^u=\oplus E_i^u$. If these subspaces are one-dimensional, then they are tangent to subgroups $W_i^u$ at the identity. Moreover, the foliation by $W_i^ux$ leaves is invariant under left translation by elements of $W_i^u$. Now, descending to the quotient $N/\Gamma$, we are interested in the density of the $W_i^ux\Gamma$ leaves. As before, we write $N_k$ for the $k$th term in the lower central series of $N$.

\begin{defn}\label{defn:irreducibility}
We say that an Anosov automorphism $L$ of a nilmanifold $N/\Gamma$ is \emph{irreducible} if, for any eigenvector $w$ of $L$ such that $w\in \cgrad{\mf{n}}{k}$, we have  
$\overline{W\Gamma}=N_k\Gamma$ in $N/\Gamma$ where 
$W$ is the one-parameter subgroup tangent to $w$. Here $\cgrad{\mf{n}}{k}$ is the $k$th term in the $L$-grading as defined in subsection \ref{subsec:automorphisms}.
\end{defn}

Note that in the case that $N/\Gamma$ is a torus that Definition \ref{defn:irreducibility} coincides with the definition of irreducibility recalled in the previous subsection. We think of the above definition as a dynamically defined irreducibility criterion because it concerns the density of leaves of a foliation defined using $L$. The following proposition gives an algebraic characterization of irreducibility. Recall that if $N$ is a nilpotent Lie group admitting a lattice $\Gamma$, then $\exp^{-1}(\Gamma)$ is a discrete subset of $\mf{n}$. Any basis of $\mf{n}$ contained in the $\Z$-span of $\exp^{-1}(\Gamma)$ has rational structure constants. Such a choice of basis gives a $\Q$-structure on $\mf{n}$ in the sense that it determines $\Q$-vector space $V$ contained in $\mf{n}$ of the same dimension as $\mf{n}$. If $L\colon N\to N$ is an automorphism that preserves $\Gamma$, then with respect to a basis contained in $\exp^{-1}(\Gamma)$, $dL$ is defined over $\Q$.

\begin{prop}\label{prop:alg_condition}
Suppose that $L\colon N/\Gamma\to N/\Gamma$ is an automorphism of a nilmanifold with sorted simple spectrum. Then $L$ is irreducible if and only if for each $k$, the induced action of $L$ on $N_k/[N_{k},N_{k}]$ is irreducible over $\Q$ with respect to the $\Q$-structure given by $\Gamma$. 
\end{prop}

Before we give the proof, we elaborate on what is meant by the statement of the proposition. If $\Gamma$ is a lattice in $N$, then $\Gamma\cap N_k$ is a lattice in $N_k$, see \autocite[Thm. 2.3 Cor. 1]{raghunathan1972discrete}. Consequently, if $L$ is a map of $N$ preserving $\Gamma$, then $L$ restricts to a map on $N_k$ preserving $\Gamma\cap N_k$. There is a quotient map $\pi\colon N_{k}\to N_{k}/[N_{k},N_{k}]$. We may then further quotient by $\pi(\Gamma\cap N_{k})$. The automorphism $L$ induces an automorphism on the resulting torus $(N_{k}/[N_{k},N_{k}])/\pi(\Gamma\cap N_{k})$. The irreducibility in the theorem is equivalent to the irreducibility for each $k$ of the automorphism on this torus in the sense mentioned in the previous subsection.

In the proof of the proposition, we use the following claim that is derived from the discussion of nilrotations in \autocite[Ch. 10]{einsiedler2013ergodic}.
\begin{lem} Suppose that $\exp(tv)$ is a one-parameter subgroup of a nilpotent group $N$ containing a lattice $\Gamma$. Then $\exp(tv)$ is dense in $N/\Gamma$ if and only if $\pi(\exp(tv))$ is dense in $(N/[N,N])/(\pi(\Gamma))$.
\end{lem}

We now turn to the proof of the proposition.

\begin{proof}[Proof of Proposition \ref{prop:alg_condition}]

The equivalence is seen by using the claim stated in the previous paragraph. Suppose that $W$ is a one-parameter subgroup tangent to $w$, which is an eigenvector of $L$ in $\cgrad{\mf{n}}{k}$.
By the lemma, $W\Gamma$ is dense in  $N_{k}/(\Gamma\cap N_k)$ if and only if $\pi(W)$ is dense in $(N_{k}/[N_{k},N_{k}])/\pi(N_k\cap \Gamma)$. 
Note that $\pi(W)$ is a one-parameter subgroup tangent to $d\pi(w)$, which is an eigenvector tangent to an eigenspace of the induced automorphism on this quotient torus. 
Using the $\Q$-structure coming from the lattice, the induced map on the quotient torus $(N_k/[N_k,N_k])/\pi(N_k\cap \Gamma)$ may be identified with a matrix $A_k\in \SL_m(\Z)$ for some $m$. Each eigenspace of $A_k$ is dense in this torus if and only the induced automorphism $A_k$ is irreducible in the sense that its characteristic polynomial does not factor over $\Q$. Thus we have obtained equivalence between irreducibility in the dynamical sense and in the algebraic sense.
\end{proof}

\begin{rem}
In the case of maps of the torus, one is able to obtain estimates on how many of the elements of $\SL_n(\Z)$ define rigid automorphisms of the torus \autocite[Prop. 3.1]{gogolev2011local}. One is able to do this by stating reducibility as a Zariski closed condition on $\SL_n(\R)$. However, in our case as the condition of sorted spectrum involves inequalities between eigenvalues it is unclear if this approach can be adapted.
\end{rem}

\section{Foliations with isometric holonomies}\label{sec:foliations_with_isometric_isometries}

In this section, we prove a rigidity result characterizing a particular type of topological foliation subordinate to the strong unstable foliation of an Anosov automorphism $L$ on $N/\Gamma$ with sorted simple spectrum.

\begin{prop}\label{prop:isometric_foliations}
Suppose that $L$ is an irreducible automorphism of a nilmanifold $N/\Gamma$ with sorted simple spectrum. Suppose that $\mc{F}$ is an $L$-invariant, continuous foliation subordinate to the $\mc{S}^{u,L}_{i}$ foliation. Moreover, suppose that $\mc{F}$ and $\mc{W}^{u,L}_{i}$ have subordinate product structure to the $\mc{S}_i^{u,L}$ foliation. Further, suppose that the $\mc{F}$ holonomy between the $\mc{W}_i^{u,L}$ leaves endowed with their right-invariant metric is an isometry. Then $\mc{F}$ coincides with $\mc{S}_{i+1}^{u,L}$.
\end{prop}

In the proof of this proposition we use the following lemma. Note that without the product structure the lemma is false. For example, consider the Reeb foliation.

\begin{lem}\label{lem:leaf_convergence}
Let $L$ be a irreducible Anosov automorphism with sorted simple spectrum of a nilmanifold $N/\Gamma$. Assume that $\mc{F}$ is a continuous foliation that subfoliates the $\mc{S}_i^{u,L}$ foliation and has product structure with the $\mc{W}_{i}^{u,L}$ foliation subordinate to the $\mc{S}_i^{u,L}$ foliation. Write $\wt{\mc{F}}$ for the lift of this foliation to $N$. Suppose that for some $x_n,y_n\in N$ that $x_n\to x$ and $y_n\to y$ and $y_n\in \wt{\mc{F}}(x_n)$. Then $y\in \wt{\mc{F}}(x)$.
\end{lem}

\begin{proof}
We work in the universal cover. Suppose we have such a pair of sequences. Pick a transversal $\tau$ to the $\mc{S}_i^{u,L}$ foliation of $N$. Note that the algebraic structure of the $\mc{S}_i^{u,L}$ foliation allows us to ensure that any leaf of $\mc{S}_i^{u,L}$ intersects $\tau$ at exactly one point.
We define a map $\Pi\colon \bigsqcup_{p\in \tau} \mc{W}_i^{u,L}(t)\to W_i^u$ as follows. First, using the subordinate product structure project along $\wt{\mc{F}}$ onto the $\mc{W}_i^{u,L}$ leaf containing $p$. Since $\mc{W}_i^{u,L}=W_i^up$, we then compose with the identification of $W_i^{u}p$ with $W_i^u$ via $n\mapsto np^{-1}$. Note that $\Pi$ is continuous due to the continuity of $\mc{F}$ and the subordinate product structure. Observe that if $x\in \mc{S}_i^{u,L}(y)$, then $\wt{\mc{F}}(x)=\wt{\mc{F}}(y)$ if and only if $\Pi(x)=\Pi(y)$ due to the product structure.  To conclude, note that $\Pi(x_n)=\Pi(y_n)$, and so  by continuity we have $\Pi(x)=\Pi(y)$.
\end{proof}

The contradiction obtained in the following proof is the same contradiction as obtained in \autocite[p. 851]{gogolev2011local}.

\begin{proof}[Proof of Proposition \ref{prop:isometric_foliations} ]
We proceed by induction and show that $\mc{F}$ is invariant under left translation by elements of the subgroups $S^u_{\dim E^u-j}$. We begin with $j=0$, and increase $j$ until we reach $j=\dim E^u-(i+1)$.

Suppose we know the result for $j$; we show it for $j+1$.
By continuity of $\mc{F}$ and the density of periodic points of $L$, it suffices to prove that $\mc{F}(x)=\mc{S}_i^{u,L}(x)$ at each periodic point $x$. For any periodic point, we may pass to a power of the dynamics so that the point is fixed. If $p\Gamma$ is a fixed point of $L$ then consider the cover $\pi\circ R_{p}\colon N\to N/\Gamma$ which sends $x\mapsto xp\Gamma$. Note that in this cover $e$ is in the fiber over $p\Gamma$. As $L(p\Gamma)=p\Gamma$, $L\colon N\to N$ is a lift of $L\colon N/\Gamma \to N/\Gamma$ with respect to this cover. These two observations show that we have reduced to the case of the periodic point $p\Gamma$ being equal to the $e\Gamma$. Using the lift to this particular cover, it now suffices to show that $\wt{\mc{F}}(e)$ is equal to  $\mc{S}_{i+1}^{u,L}(e)$.

Suppose that $s\in S^u_{\dim E^u-(j+1)}$. If $s\in \wt{\mc{F}}(e)$, then we are done. For the sake of contradiction, suppose not. 
By Proposition \ref{prop:isometry_description}, which characterizes the isometries between leaves of the $\mc{W}_i^u$ foliation, there exists $w\in W_i^u$ such that 
\begin{equation}\label{eq:what_holonomy_is}
H^{\wt{\mc{F}}}_{e,s}(x)=x w s.
\end{equation}

By Proposition \ref{prop:alg_condition}, the assumption of irreducibility implies  that
$S_i^u\Gamma\subseteq \overline{W_i^u\Gamma}$ in $N/\Gamma$. Thus we may 
fix a sequence $b_i\in W_i^u$ and $\gamma_i\in \Gamma$ such that $b_i\gamma_i^{-1}\to ws$.

Now, we have that $b_iw s\in \wt{\mc{F}}(b_i)$ by equation (\ref{eq:what_holonomy_is}). Thus 
\[
b_i ws\in \wt{\mc{F}}(b_i).
\]
Note that by the definition of the commutator\footnote{We define the commutator by $[g,h]=ghg^{-1}h^{-1}$.},
\[
b_i ws =[b_i,ws]ws b_i,
\]
and that $[b_i,ws]\in S^u_{\dim E^u-j}$, by the Carnot grading on $N^u$.  Since $\wt{\mc{F}}$ is $S^u_{\dim E^u-j}$-invariant under left multiplication, we have
\[
ws b_i \in \wt{\mc{F}}(b_i).
\]
Consequently by right-invariance of $\wt{\mc{F}}$ under the action of $\Gamma$,
\[
ws b_i\gamma_i^{-1}\in \wt{\mc{F}}(b_i\gamma_i^{-1}).
\]
As $b_i\gamma_i^{-1}\to ws$, and left multiplication is continuous, we find that 
\[
ws b_i\gamma_i^{-1}\to (ws)^2.
\]
So, by the Lemma \ref{lem:leaf_convergence}, we see that 
\[
(ws)^2\in \wt{\mc{F}}(ws)=\wt{\mc{F}}(e).
\]

By once again using the Carnot structure and commutators we see that $wsws=s'w^2s^2$, for some $s'\in S^u_{\dim E^u-j}$. So, again using invariance under left multiplication, we see that $w^2s^2\in \wt{\mc{F}}(e)$. By repeating the argument from the point where we chose the sequence $b_i$, we obtain that $w^{2^m}s^{2^m}$ is in $\wt{\mc{F}}(e)$ for all $m\ge 0$.

By considering the weak and strong distances, we show that this leads to a contradiction.
By Lemma \ref{lem:distance_growth}, there exist $C,D>0$ such that
\[
d_{\mc{S}_i}(e,w^{2^m}s^{2^m})\le 2^mC
\]
and
\[
d_{\mc{W}_i}(e,w^{2^m}s^{2^m})=2^mD.
\]
We will obtain a contradiction by applying $L^{-1}$.
Fix some small $\epsilon>0$ and let $c_{\epsilon}(m)$ be the minimum $j\ge 0$ such that $\epsilon^2<d_{\mc{S}_i}(e,L^{-j}(w^{2^m}s^{2^m}))<\epsilon$. 
We now estimate the weak and strong distances of the points $L^{-c_{\epsilon}(m)}(w^ms^m)$ from $e$. By Claim \ref{lem:horizontal_distance_growth},  we see that the weak distance is contracted at a rate of exactly $\abs{\lambda_i^u}$. The strong distance is contracted by a rate of at least $\abs{\lambda^u_{i+1}}$, as the norm of the differential of $L$ restricted to $S^u_{i+1}$ is bounded below by $\abs{\lambda_{i+1}^u}$. Consequently, we see that 
\[
\abs{\lambda_{i+1}^u}^{c_{\epsilon}(m)}\le 2^mC/\epsilon^2.
\]
Consequently, for each $m$, by Claim \ref{lem:horizontal_distance_growth} and the inequality,

\begin{align*}
d_{\mc{W}_i}(L^{-c_{\epsilon}(m)}(e),L^{-c_\epsilon(m)}(w^ms^m))&=\abs{\lambda^u_{i}}^{-c_{\epsilon(m)}}2^mD\\
&\ge \abs{\lambda_i^u}^{-c_{\epsilon}(m)}\abs{\lambda_{i+1}^u}^{c_{\epsilon}(m)}\frac{D}{C}\epsilon^{2}.
\end{align*}
Observe that as $m\to \infty$, that $c_{\epsilon}(m)\to \infty$ as well, so this lower bound is going to $\infty$. This is impossible however: we will show that it violates the continuity of the $\wt{\mc{F}}(e)$ foliation.

We claim that the set
\[
K=\{x\in \wt{\mc{F}}(e)\mid d_{\mc{S}_i}(e,x)\le \epsilon\}
\]
is compact. To see this, note that the $\wt{\mc{F}}(e)$ foliation and the $\mc{S}_{i+1}^{u,L}$ foliation both have subordinate global product structure with the $\mc{W}_i^{u,L}$ foliation to the $\mc{S}_i^{u,L}$. Consequently, there is a well-defined homeomorphism $\Pi\colon S_{i+1}^{u,L}\to \wt{\mc{F}}(e)$ given by $s\mapsto W_{i}^{u,L}s\cap \wt{\mc{F}}(e)$. Note that $\Pi(s)=ws$ for some $w\in W_i^{u}$, so, by Lemma \ref{lem:vh_distance_invariance}, $\Pi$ preserves strong distances from $e$. Consequently, 
\[
K=\Pi(\{x\in S_{i+1}^{u}\mid d_{\mc{S}_i}(e, x)\le \epsilon\}),
\]
is compact. Because weak distance varies continuously and $K$ is compact, the function $x\mapsto d_{\mc{W}_i}(e,x)$ is bounded on $K$. But this contradicts the existence of the points $L^{-c(m)}(s^mw^m)\in \wt{\mc{F}}(e)$ that have strong distance from $e$ of size at most $\epsilon$, but for large $m$ have unbounded weak distance from $e$. Having reached this contradiction, we see that $s\in \wt{\mc{F}}(e)$ and we are done.
\end{proof}

\section{Proof of main rigidity theorem}

In this section, we prove the sufficiency of the condition in Theorem \ref{thm:rigidity}.

\begin{thm}\label{thm:sufficiency}
Suppose that $L$ is an Anosov automorphism of a nilmanifold $N/\Gamma$ that is irreducible and has sorted simple Lyapunov spectrum. Then there exists a $C^1$ neighborhood $\mc{U}$ of $L$ in $\Diff^{1+}_{vol}(N/\Gamma)$ such that if $f\in \mc{U}$ and the Lyapunov spectrum of $f$ with respect to volume coincides with that of $L$ and $h_f$ is a conjugacy between $f$ and $L$, then $h_f$ is $C^{1+}$.
\end{thm}

We begin with a lemma.

\begin{lem}\label{lem:F_is_isometric}
Let $N/\Gamma$ and $L$ be given as in Theorem \ref{thm:sufficiency}. Suppose that $\mc{F}$ is a continuous, $L$-invariant foliation subordinate to the $\mc{S}^{u,L}_{i}$ foliation, that $\mc{F}$ and $\mc{W}_i^{u,L}$ have subordinate product structure, and that $\mc{F}$ has uniformly $C^{1}$ holonomies between the $\mc{W}_i^{u,L}$ leaves. Then the holonomies of $\mc{F}$ between the leaves of $\mc{W}_i^{u,L}$ are isometries in the induced metric on the $\mc{W}_i^{u,L}$ leaves.
\end{lem}

\begin{proof}
Write $H^\mc{F}_{a,b}$ for the holonomy of $\mc{F}$ between $\mc{W}_i^{u,L}(a)$ and $\mc{W}_i^{u,L}(b)$. As $\mc{F}$ is $L$-invariant, this holonomy satisfies 
\[
H_{a,b}^{\mc{F}}=L^{n}\mid_{\mc{W}_i^{u,L}(b_n)}\circ H_{a_n,b_n}^{\mc{F}}\circ L^{-n}\mid_{\mc{W}_i^{u,L}(a)},
\]
where $a_n=L^{-n}a, b_n=L^{-n}b$.

It suffices to show that that $\|DH^{\mc{F}}_{a,b}\|=1$, as we are working with maps of $1$-manifolds. The differentials are conformal, so the norm of a composition is the product of norms. In particular,
\[
\|DH^{\mc{F}}_{a,b}\|=\|DL^{n}\|\|DH^{\mc{F}}_{a_n,b_n}\|\|DL^{-n}\|.
\]

As we are regarding $\mc{W}_i^u$ with its right-invariant metric, the norm of $DL$ is constant, so the first and third terms above multiply to $1$. Thus we need only show that $\|DH^{\mc{F}}_{a_n,b_n}\|\to 1$ as $n\to \infty$. Pass to a subsequence so that $a_n,b_n\to c$ in $N/\Gamma$. Then as the holonomies are uniformly $C^{1}$, we see that $\|DH_{a_n,b_n}^{\mc{F}}\|$ converges to $\|DH_{c,c}^{\mc{F}}\|=1$ because $H_{c,c}^{\mc{F}}$ is the identity. Thus $\|DH_{a,b}^{\mc{F}}\|=1$. The result follows.
\end{proof}

\begin{rem}
If the foliation $\mc{W}_i^u$ had higher dimensional leaves and we assumed that $DL$ is conformal on $\mc{W}_i^u$, then the proof of Lemma \ref{lem:F_is_isometric} still works and we get the same conclusion.
\end{rem}

We now proceed to the proof of the theorem.

\begin{proof}[Proof of Theorem \ref{thm:sufficiency}]
By Proposition \ref{prop:slowtoslow}, there exists a $C^1$ neighborhood $\mc{U}$ of $L$ in $\Diff^{1+}_{\vol}(M)$ such that if $f\in \mc{U}$ and $h_f$ is a conjugacy between $f$ and $L$, and if, for some $i$, $h_f(\mc{S}_i^{u,f})=\mc{S}_i^{u,L}$, then there exists a continuous foliation $\mc{W}_{i}^{u,f}$ satisfying the properties mentioned in the conclusion of Proposition \ref{prop:slowtoslow}.

Suppose that $f\in \mc{U}$ and that the volume Lyapunov spectrum of $f$ coincides with the volume Lyapunov spectrum of $L$. We will prove the claim about $f$ by induction.
However, before proceeding to the induction we make the following observation. 
\begin{lem}\label{lem:h_is_C1}
If $h_f(\mc{W}_i^{u,f})=\mc{W}_i^{u,L}$, then $h_f$ is uniformly $C^{1+}$ along $\mc{W}_i^{u,f}$.
\end{lem}

\begin{proof}
If $h_f(\mc{W}_i^{u,f})=\mc{W}_i^{u,L}$, 
then $h_f$ intertwines the action of $f$ and $L$ on the $\mc{W}_i^{u,f}$ and $\mc{W}_i^{u,L}$ foliations. Both are expanding foliations, as elements of $\mc{U}$ have simple Mather spectrum. By the Pesin entropy formula,
\[
h_f(\vol)=\int_{N/\Gamma} \sum_{1\le i\le \dim E^u} \lambda_i^u\,d\vol,
\]
which coincides with the entropy of $L$ against volume. Volume is the measure of maximal entropy for $L$. Consequently as $f$ and $L$ have the same volume entropy, $\vol$ is also the unique measure of maximal entropy for $f$.
Thus $(h_f)_*\vol=\vol$ as $\vol$ is the unique measure of maximal entropy for $f$ and $L$ and a conjugacy intertwines the measures of maximal entropy.
We next claim that the disintegration of volume along $\mc{W}_i^{u,L}$ and $\mc{W}_i^{u,f}$ leaves is absolutely continuous. The case of $\mc{W}_i^{u,L}$ is immediate by Fubini's theorem as $\mc{W}_i^{u,L}$ is analytic. We now explain why the disintegration along $\mc{W}_i^{u,f}$ is absolutely continuous.
 First, note that if $\xi$ is an increasing measurable partition subordinate to the $\mc{W}_i^{u,L}$ foliation then $h_f^{-1}(\xi)$ is an increasing measurable partition subordinate to the $\mc{W}_i^{u,f}$ foliation as $(h_f)_*(\vol)=\vol$. Consequently, we have the following equality of conditional entropies: 
\[
H_{\vol}(f^{-1}(h_f^{-1}(\xi))\mid h_f^{-1}(\xi))=H_{\vol}(L^{-1}\xi\mid \xi)=\ln \abs{\lambda_i^{u}}.
\]
But as the volume spectrum of $f$ is the same as the volume spectrum of $L$, we see that 
\[
\int \ln \|Df\mid_{\mc{W}_i^{u,f}}\|\,d\vol=\ln \abs{\lambda_i^{u}}
\]
as well. Consequently, the hypotheses of Lemma \ref{lem:pesin} are satisfied and so the disintegration of volume along the $\mc{W}_i^{u,f}$ foliation is absolutely continuous.
Then, by Lemma \ref{lem:regularitycriterion}, we conclude that $h_f$ is uniformly $C^{1+}$ along $\mc{W}_i^{u,f}$. 
\end{proof}

We now proceed by induction to show that 
$h_f(\mc{S}_i^{f,u})=\mc{S}^{u,L}_{i}$; i.e., that $h_f$ carries strong foliations to strong foliations.
 We induct on $1\le i \le \dim E^u$ beginning with $i=1$.
 In the case that $i=1$, this is the statement that a conjugacy carries unstable manifolds to unstable manifolds, which is verified in Proposition \ref{prop:unstabletounstable}.
Suppose now that the claim holds for $i$. Then as $f\in \mc{U}$ and the induction hypothesis, we see that there exists a foliation $\mc{W}_i^{u,f}$ such that $h_f(\mc{W}_i^{u,f})=\mc{W}_i^{u,L}$ satisfying the conclusion of Proposition \ref{prop:slowtoslow}. By Lemma \ref{lem:h_is_C1}, $h_f$ is uniformly $C^{1+}$ along $\mc{W}_i^{u,f}$.

Let $\mc{F}$ denote the image of $\mc{S}_{i+1}^{u,f}$ by $h_f$. Then $\mc{F}$ is a subfoliation of $\mc{S}_i^{u,L}$, by the induction hypothesis. As $h_f(\mc{W}_i^{u,f})=\mc{W}_i^{u,L}$, $\mc{F}$ and 
$\mc{W}_i^{u,L}$ have subordinate product structure to the $\mc{S}_i^{u,L}$ foliation.
Further, we claim that the holonomy of $\mc{F}$ between $\mc{W}_i^{u,L}$ leaves is uniformly $C^{1+}$. The holonomy $H^{\mc{F}}$ is the composition $h_f\circ H^{\mc{W}^{u,f}_{i+1}}\circ h_f^{-1}$. The conjugacy $h_f$ restricted to $\mc{W}_i^{u,L}$ is uniformly $C^{1+}$ by the previous discussion.  The holonomies of the fast foliation $\mc{S}^{u,f}_{i+1}$ between leaves of the $\mc{W}_i^{u,f}$ foliation are uniformly $C^{1+}$ by Proposition \ref{prop:slowtoslow}. Thus $\mc{F}$ satisfies the hypotheses of Lemma \ref{lem:F_is_isometric} and so $\mc{F}$ has isometric holonomies between $\mc{W}_i^{u,L}$ leaves. Consequently, $\mc{F}$ satisfies the hypotheses of Proposition \ref{prop:isometric_foliations}, which implies $\mc{F}=\mc{S}^{u,L}_{i+1}$. Thus $h_f(\mc{S}^{u,f}_{i+1})=\mc{S}^{u,L}_{i+1}$ and the induction holds.

Note that at each step in the induction that we concluded that $h_f$ is uniformly $C^{1+}$ along $\mc{W}_i^{u,f}$.
This shows that for $f\in \mc{U}$ with the same volume spectrum as $L$ that the map $h_f$ is $C^{1+}$ along $\mc{W}_i^{u,f}$ for each $1\le i\le \dim E^u$. To conclude that $h_f$ is $C^{1+}$ on the full unstable manifold $\mc{S}_1^{u,f}$, we now appeal to Journ\'e's lemma:

\begin{lem}\autocite{journe1988regularity} Let $\mc{F}_s$ and $\mc{F}_u$ be two continuous transverse foliations with uniformly $C^{1+}$ leaves. If $f$ is uniformly $C^{1+}$ along the leaves of $\mc{F}_s$ and $\mc{F}_u$ then $f$ is $C^{1+}$.
\end{lem}

We now apply Journ\'e's lemma inductively. The foliations $\mc{S}^{u,f}_{i+1}$ and $\mc{W}_i^{u,f}$ are transverse subfoliations of $\mc{S}^{u,f}_i$. So, if $h_f$ is $C^{1+}$ along both, then, by the lemma, $f$ is $C^{1+}$ along $\mc{S}^{u,f}_i$. Proceeding inductively from strongest to weakest, we see that $h_f$ is $C^{1+}$ along the full unstable manifold. Repeating the argument for the stable manifold gives the full result.
\end{proof}

\section{Necessity of irreducibility and sorted spectrum for local rigidity}\label{sec:necessity}

In this section, we establish through constructions a necessary condition for Lyapunov spectrum rigidity in the case of simple spectrum. We will frequently consider a nilmanifold $N/\Gamma$ as well as the quotient nilmanifold $N/Z(N)/\pi(\Gamma)$, which we denote by $\underline{N}/\underline{\Gamma}$. As elsewhere, $Z(N)$ denotes the center of $N$. If $L$ is an Anosov automorphism of $N/\Gamma$, we denote by $\underline{L}$ the induced map on $\underline{N}/\underline{\Gamma}$.
To show necessity, we produce perturbations of $L$ with the same periodic data that are not even Lipschitz conjugate to $L$. A volume-preserving map with the same periodic data as $L$ has the same volume spectrum as $L$ by Proposition \ref{prop:periodic_approximation}. This implies the necessity of the condition in Theorem \ref{thm:rigidity} on volume Lyapunov spectrum rigidity.
The proof of necessity proceeds by induction. The base case of the induction is the claim that if the induced automorphism $\underline{L}$ of $\underline{N}/\underline{\Gamma}$ is irreducible and has sorted spectrum but $L$ does not, then $L$ is not periodic data rigid.
The induction step shows that if an automorphism is not rigid, then a central extension of this automorphism is also not rigid. By considering iterated central extensions, we reduce to the base case.

The organization of this section is as follows. First, we give explicit constructions in the base case depending on whether the automorphism is not reducible or fails to have sorted spectrum.
The approach is an extension to nilmanifolds of the perturbative technique studied by Gogolev \autocite{gogolev2008smooth} and de la Llave \autocite{de1992smooth} in the case of the torus.
 The general idea is to shear a fast unstable direction into a slower unstable direction.
 After giving the constructions in the base case, we give a separate construction for the induction step. In the final section, we conclude.

\subsection{Non-sorted spectrum}\label{subsec:foo}

In this section we show that if $L\colon N/\Gamma\to N/\Gamma$ is an Anosov automorphism with unsorted simple spectrum and $\underline{L}\colon \underline{N}/\underline{\Gamma}\to \underline{N}/\underline{\Gamma}$ is sorted, then $L$ is not rigid.

\begin{prop}\label{prop:out_of_order_reducible_prop}
Suppose that $L\colon N/\Gamma\to N/\Gamma$ is an Anosov automorphism with simple spectrum of a nilmanifold $N/\Gamma$. Suppose that the induced action on $\mf{n}$ has an unstable eigenvector $w\in \mf{z}$, where $\mf{z}$ is the center of $\mf{n}$, and another unstable eigenvector $u\notin \mf{z}$. Write $\lambda_w$ and $\lambda_u$ for the eigenvalues of $u$ and $w$. If $\abs{\lambda_u}>\abs{\lambda_w}$, then $L$ is not Lyapunov spectrum or periodic data rigid. Indeed, there exist arbitrarily $C^{\infty}$-small perturbations of $L$ with the same periodic data so that a conjugacy between $L$ and the perturbation need not even be Lipschitz.
\end{prop}

Before we begin the proof, we outline the approach. We construct a family of perturbations of $L$ obtained by shearing the base dynamics into the slow direction, $w$, in the fiber. For each element in the family, we then obtain an explicit equation for a conjugacy between that element and $L$. By analyzing the resulting equation, we then obtain a necessary and sufficient condition for the conjugacy to be Lipschitz. We then show that this necessary condition for the conjugacy to be Lipschitz does not hold on the entire family, which proves the proposition. This approach is originally due to de la Llave \autocite{de1992smooth} and was extended significantly by Gogolev in \autocite{gogolev2008smooth}. Both studied maps of the torus. The criterion for regularity appearing below is due to Gogolev, see \autocite[Prop. 1]{gogolev2008smooth}.

\begin{proof}[Proof of Proposition \ref{prop:out_of_order_reducible_prop}.]
The map $L$ descends to a map $\underline{L}$ on $\underline{N}/\underline{\Gamma}$. 
If $x\in N/\Gamma$, we write $\underline{x}$ to denote the image of $x$ in $\underline{N}/\underline{\Gamma}$. We will also use $\underline{x}$ to denote an element of $\underline{N}/\underline{\Gamma}$ even if we have not introduced an element $x$.

Recall that $Z$ acts on $N/\Gamma$ on the left and this action preserves the structure of the fibration $\mathbb{T}^n\to N/\Gamma\to (N/Z)/\pi(\Gamma)$. Consequently, for an element $z\in Z$ and a point $x\in N/\Gamma$, we may consider the point $x+z$. 
If $\phi\in C^{\infty}(\underline{N}/\underline{\Gamma},\R)$, we define a map $I_\phi\colon N/\Gamma\to N/\Gamma$ via $x\mapsto x+\phi(\underline{x}) w$.  We similarly define the perturbation $L_{\phi}\coloneqq L(x)+\phi(\underline{x})w$. Finally, we observe that for $t\in \R$ and $x\in N/\Gamma$, that $L(x+t w)=L(x)+\lambda_w t w$.

\begin{lem}
The perturbation $L_\phi$ has the same periodic data as $L$.
\end{lem}
\begin{proof}
Consider the differential of $L_\phi$ when viewed in charts that trivialize the principal bundle structure of $N/\Gamma$. In such charts, $L_\phi$ is a map $B_1\times \mathbb{T}^n\to B_2\times \mathbb{T}^n$ where $B_i$ is an open disk in $\R^k$ and $k=\dim \underline{N}/\underline{\Gamma}$. As $L_\phi$ is a bundle map, we may choose the chart so that the differential is of the form:
\[
\begin{bmatrix}
D\underline{L} & 0 \\
D\phi(\underline{x})w & L\mid_{Z(N)}
\end{bmatrix}.
\]
We see immediately from the block form of this matrix that $L_{\phi}$ has the same periodic data as $L$.
\end{proof}

\begin{lem}\label{lem:I_psi_conjugacy}
If $\phi\in C^{\infty}(\underline{N}/\underline{\Gamma},\R)$ and $\psi\in C^0(\underline{N}/\underline{\Gamma},\R)$  satisfy the cohomological equation
\begin{equation}\label{eq:conj_criterion}
\phi(\underline{x})+\psi(\underline{L}\,\underline{x})=\lambda_w\psi(\underline{x}),
\end{equation}
for all $x\in \underline{N}/\underline{\Gamma}\,$,
then $I_\psi$ is a conjugacy between $L_{\phi}$ and $L$.
\end{lem}

\begin{proof}
We check by computation that $I_\psi\circ L_{\phi}=L\circ I_\psi$. Suppose that $x\in N/\Gamma$, then $I_\psi\circ L_{\phi}$ sends 
\[
x\mapsto L(x)+\phi(\underline{x})w \mapsto L(x)+\phi(\underline{x})w+\psi(\underline{L}\,\underline{x})w= I_\psi\circ L_{\phi}(x).
\]
On the other hand, $L\circ I_\psi(x)$ is 
\[
L(x)+\lambda_w\psi(\underline{x})w.
\]
Thus if the stated condition holds, then $I_\psi$ is such a conjugacy.
\end{proof}

In fact, we can write down an explicit function $\psi$ satisfying equation (\ref{eq:conj_criterion}). The form of $\psi$ can be guessed by using the relation appearing in Lemma \ref{lem:I_psi_conjugacy} as a recurrence. Define
\begin{equation}
\mc{J}(\phi)=\lambda_w^{-1}\sum_{k\ge 0}\lambda_w^{-k} \phi(\underline{L}\,\underline{x}),
\end{equation}
so that $I_{\mc{J}(\phi)}$ provides a conjugacy between $L$ and $L_{\phi}$. Note that $\mc{J}(\phi)$ is well-defined as a continuous function as this series is absolutely convergent.

\begin{lem}
The vector $u$ defines a vector field $U$ on $\underline{N}/\underline{\Gamma}$. Let $\mc{U}$ be the foliation of $\underline{N}/\underline{\Gamma}$ tangent to the integral curves of $U$. For $\phi\in C^{\infty}(\underline{N}/\underline{\Gamma},\R)$, $\mc{J}(\phi)$ is Lipschitz along $\mc{U}$ if and only if the following two sums converge and are equal in the sense of distributions,
\begin{equation}\label{eqn:lipschitz_condition}
\sum_{k\ge 0} \lambda_{w}^{-k}\lambda_u^k \phi_u\circ \underline{L}^k=-\lambda_u^{-1}\sum_{k<0}\lambda_w^{-k}\lambda_u^{k}\phi_u\circ \underline{L}^k,
\end{equation}
where $\phi_u=U(\phi)$ is the derivative of $\phi$ in the direction of $U$.
\end{lem}

\begin{proof}
In the proof we write $\psi$ for $\mc{J}(\phi)$ for both clarity and convenience.
First, suppose that $\psi$ is Lipschitz along $\mc{U}$. Then for Lebesgue-a.e.\@ point $p$ in $\underline{N}/\underline{\Gamma}$,  $\phi_u$ exists at $p$, and hence by differentiating equation (\ref{eq:conj_criterion}), there is an $\underline{L}$-invariant set of full volume such that 
\[
\phi_u+\lambda_u \psi_u\circ \underline{L}=\lambda_w \psi_u
\]
on this set. This relation implies that almost everywhere
\[
\psi_u=-\frac{1}{\lambda_u}\phi_u\circ \underline{L}^{-1}+\frac{\lambda_w}{\lambda_u}\psi_u\circ \underline{L}^{-1}.
\]
By similarly using this relation as a recurrence, we obtain that 
\begin{equation}\label{eq:psi_u_expression}
\psi_u=-\lambda_u^{-1} \sum_{k<0} \pez{\frac{\lambda_u}{\lambda_w}}^k\phi_u\circ \underline{L}^k,
\end{equation}
almost everywhere.
As $\abs{\lambda_u}>\abs{\lambda_w}$, the series above converges in $C^0$ sense. As $\psi$ is Lipschitz along $\mc{U}$, it is equal to the integral of its derivative by $U$. The foliation $\mc{U}$ is analytic, so we see that $\psi$ is differentiable with derivative $\psi_u$, and  $\psi_u$ satisfies equation \eqref{eq:psi_u_expression} along a.e.\@ $\mc{U}$ leaf. In particular, this implies that the distributional derivative of $\psi$ along almost every $U$ line is given by pairing with $\psi_u$ as in equation \eqref{eq:psi_u_expression}. This implies that the distributional derivative of $\psi$ in the direction $U$ is regular, and, in particular, is given by pairing with the expression in equation \eqref{eq:psi_u_expression}.

We may compute the distributional derivative of $\psi$ in another way as well. In particular, by definition,
\[
\psi=\sum_{k\ge 0} \lambda_w^{-k} \phi\circ \underline{L}^k.
\]
To find the distributional derivative of $\psi$ in the direction $U$, we differentiate term by term; hence, in the sense of distributions,
\[
\psi_u  =\sum_{k\ge 0}\lambda_{w}^{-k} \lambda_u^k \phi_u\circ \underline{L}^k.
\]
Thus the claimed equality holds. This establishes the implication in the theorem.

Next, suppose that the stated sums converge in the sense of distributions and are equal. The distribution given by pairing with 
\[
\psi=\sum_{k\ge 0} \lambda_w^{-k}\phi\circ \underline{L}^k,
\]
has distributional derivative  in direction $U$ given by the sum of the distributions
\[
\sum_{k\ge 0} \lambda_w^k \lambda_u^k \phi_u\circ \underline{L}^k.
\]
By assumption this is equal to the distribution
\[
-\lambda_u^{-1}\sum_{k<0}\lambda_w^{-k}\lambda_u^k\phi_u\circ \underline{L}^k.
\]
However, this distribution is regular as $\abs{\lambda_u}>\abs{\lambda_w}$, and is equal to pairing with some function $\omega\in C^0$. Hence the distributional derivative of $\psi$ is given by pairing with $\omega$, i.e. $\psi$ is weakly differentiable along $U$ with weak derivative $\omega$. Note that $\omega$ is in $C^0$, as is $\psi$. Hence a standard argument implies that $\psi$ is Lipschitz in the direction $U$ with Lipschitz constant depending on $\|\omega\|_{C^0}$.
\end{proof}

The proof of Proposition \ref{prop:out_of_order_reducible_prop} is then finished by the following lemma.

\begin{lem}
In any neighborhood of $0$ in $C^{\infty}(\underline{N}/\underline{\Gamma},\R)$, there exists a function $\phi$ violating equation (\ref{eqn:lipschitz_condition}).
\end{lem}

\begin{proof}
It suffices to consider the case that the image of $u$ is central in $\mf{n}/\mf{z}$.
We may make this reduction by the following means. If $N'/\Gamma'$ is a nilmanifold fibering over $\underline{N}'/\underline{\Gamma}'$, then we can pullback a function $\phi$ on $\underline{N}'/\underline{\Gamma}'$ to a function on $N'/{\Gamma'}$. Suppose that $\underline{\phi}$ is a function on $\underline{N}'/\underline{\Gamma}$ that fails to satisfy equation (\ref{eqn:lipschitz_condition}). Then there exists a function $\underline{\psi}$ on $\underline{N}'/\underline{\Gamma}'$ that pairs to different things with each side of equation (\ref{eqn:lipschitz_condition}). Denote by $\phi$ and $\psi$ the pullbacks of these functions to $N'/\Gamma'$. We claim that $\phi$ does not satisfy equation (\ref{eqn:lipschitz_condition}). To see this note that if $\underline{\rho}$ and $\underline{\omega}$ are two functions on $\underline{N}'/\underline{\Gamma}'$ then the pairing of $\underline{\rho}$ with $\underline{\omega}$ is equal to the pairing of $\rho$ with $\omega$.

We now assume that $u$ is tangent to the central direction in $\underline{N}/\underline{\Gamma}$. As is standard, $L^2(\underline{N}/\underline{\Gamma},\mathbb{C})$ is a unitary representation of the group $N/Z$, in which $\exp(u)$ is central. Note that as the subgroup tangent to $u$ is central in $N/Z$ that $\exp(u)$ acts inside of irreducible representations by multiplication.
Pick a non-trivial irreducible representation $V_{\gamma}\subseteq L^2(\underline{N}/\underline{\Gamma},\mathbb{C})$ as well as a $C^{\infty}$ vector $\phi\in V_{\gamma}$ on which $u$ acts nontrivially. Note that there exists such a vector as there are functions on $\underline{N}/\underline{\Gamma}$ that are not constant in the $u$ direction.

 Let $U^t$ be the flow given by left translation by $\exp{ut}$. Observe that as $u$ is central inside $V_{\gamma}$ that $U^t$  acts by multiplication by $e^{i\lambda_\gamma}$ for some $\lambda_{\gamma}\in \R\setminus \{0\}$.  Suppose now that $\phi$ is a smooth function in $V_{\gamma}$. Then
\[
\phi\circ U^t=e^{i\lambda_{\gamma}t}\phi.
\]
Observe that 
\[
\phi\circ L\circ U^t=\phi\circ U^{\lambda_u t}\circ L.
\]
Thus $\phi\circ L$ lies in a representation $V_{\gamma'}$ where  $u$ acts by $\lambda_{\gamma'}=\lambda_u\lambda_{\gamma}$.  Note that $\phi_u=\lim_{t\to 0} (\phi\circ U^t-\phi)/t$ lies within $V_{\gamma}$ as $\phi$ is smooth. Similarly, by applying $L^k$, we obtain a function $\phi^{(k)}$ on which $U^t$ acts by multiplication by $e^{i\lambda_{\gamma}\lambda_u^kt}$. For $j$ and $k$ such that $\abs{j-k}$ is sufficiently large, $\phi^{(j)}$ and $\phi^{(k)}$ are orthogonal as
\[
\langle \phi^{(j)},\phi^{(k)}\rangle=\langle U^t\phi^{(j)},U^t\phi^{(k)}\rangle=\langle e^{i\lambda_{\gamma}\lambda_u^jt}\phi,e^{i\lambda_{\gamma}\lambda_u^kt}\psi\rangle=e^{i\lambda_{\gamma}\lambda_u^j t}e^{-i\lambda_{\gamma}\lambda_u^k t}\langle \phi,\phi\rangle
\]
must be constant in $t$.

Observe now that if we evaluated equation (\ref{eqn:lipschitz_condition}) on $\phi$, that the two distributions in that equation are different. The functions $\phi\circ L^k$ for large positive and negative $k$ are orthogonal by the discussion above. 
Consequently, in equation (\ref{eqn:lipschitz_condition}), 
when we pair with $\phi\circ L^k$ for sufficiently large $k$ the left hand distribution gives a non-zero quantity, while the right hand gives a zero quantity. 
By taking the real parts we obtain a function in $C^{\infty}(\underline{N}/\underline{\Gamma},\R)$ with the same property. As the relation in equation (\ref{eqn:lipschitz_condition}) is linear, if it fails for $\phi$ it fails for $\epsilon \phi$ as well, and so the needed result holds in any neighborhood of $0$ in $C^{\infty}(\underline{N}/\underline{\Gamma},\R)$.
\end{proof}

This finishes the proof of Proposition \ref{prop:out_of_order_reducible_prop}, as we have now shown that there exist arbitrarily small $\phi$ such that $L_{\phi}$ has the same periodic data as $L$ and is not Lipschitz conjugate to $L$.
\end{proof}

\subsection{Lack of irreducibility}

In this section we show the following.

\begin{prop}\label{prop:irreducibility_fails}
Suppose that $L\colon N/\Gamma\to N/\Gamma$ is an Anosov automorphism with simple spectrum. If the restricted map $L\mid_Z$, where $Z$ is the center of $Z$, is reducible with respect to the $\Q$-structure given by $\Gamma$, then $L$ is not periodic data rigid.
\end{prop}

This proposition follows because if the action of $L$ is reducible, the map \[\pi\colon N/\Gamma\to \underline{N}/\underline{\Gamma}\] naturally decomposes into a sequence of two fiber bundles each with torus fiber. We then shear a fast unstable direction tangent to one of these bundles into a slower unstable direction tangent to the other just as we did in Proposition~\ref{prop:out_of_order_reducible_prop}.

\begin{proof}[Proof of Proposition \ref{prop:irreducibility_fails}.]
If $L\mid_Z\colon Z\to Z$ is reducible with respect to the $\Q$-structure coming from $\Gamma\cap Z$ and $L\mid_Z$ has simple spectrum, then the characteristic polynomial of $L$ splits over $\Q$ into two distinct factors $p$ and $q$ with no common roots. Let $V_p$ be the subspace of $\mf{z}$ tangent to the eigenvectors of $L\mid_Z$ with eigenvalues coming from $p$ and $V_q$ by the analogous subspace for $q$. Both of these subspaces are rational. Consequently, $V_q\cap \Gamma$ is a lattice in $V_q$.

Let $\overline{V_q}$ denote the image of $V_q$ in $Z/(Z\cap \Gamma)$. 
In particular, observe that $\overline{V_q}$ is a torus. This torus acts freely and faithfully on $N/\Gamma$ by translation. Note in addition that the map $L$ commutes with the action of $\overline{V_q}$ up to an element of $\overline{V_q}$ because $\overline{V_q}$ is invariant under $L$. Consequently, if we quotient by the action of $\overline{V_q}$, we obtain  that $N/\Gamma$ fibers over some manifold $(N/\overline{V_q})/(\pi(\Gamma))$ with torus fiber and that the map $L$ descends to the quotient. 

We may now repeat the proof of Proposition \ref{prop:out_of_order_reducible_prop}. Suppose that the action of $L$ on $V_q$ has a larger unstable eigenvalue $\lambda$ than any eigenvalue of the action on $V_p$. Let $w$ be a vector tangent to the subspace associated to $\lambda$. 
Using $w$ we can shear $\overline{V_p}$ into $\overline{V_q}$ as in the proof of Proposition \ref{prop:out_of_order_reducible_prop}, and repeat the argument there to obtain that $L$ is not Lyapunov rigid.
\end{proof}

\subsection{Lack of rigidity persists in extensions}\label{subsec:bar}

If one has an automorphism $L\colon N/\Gamma\to N/\Gamma$ that is not Lyapunov spectrum rigid, then one may wonder if there exists an Anosov automorphism $L'\colon N'/\Gamma'\to N'/\Gamma'$ and a natural algebraic quotient map $\pi\colon N'/\Gamma'\to N/\Gamma$ such that $\pi\circ L'=L\circ \pi$ and $L'$ is Lyapunov spectrum rigid.  The content of the following proposition is that this cannot happen. The value of the proposition is that it allows for Lyapunov spectrum rigidity to be studied inductively.

The construction that follows is closely related to the homotopy theoretic approach developed by Gogolev, Otaneda, and Rodriguez Hertz in \autocite{gogolev2015new}. For an automorphism $A$ of a Lie group $G$, they study $A$-maps of principal $G$-bundles, which  means that if $F\colon E\to F$ is a map of principal $G$-bundles, then $F(x.g)=F(x).A(g)$. Note that any Anosov automorphism $N/\Gamma\to N/\Gamma$ is an $A$-map of $N/\Gamma$ when viewed as a map of principal torus bundles. In this case, the map $A$ is the map restricted to the torus fiber through $e\Gamma$, which is $Z/(Z\cap \Gamma)$. The basic theory of such $A$-maps is developed clearly in \autocite{gogolev2015new} and so we do not repeat the development here. We recall one result, Proposition 4.4., which gives that if $A,B\in \Aut(G)$, and $f$ is an $A$ map and $g$ is a $B$ map, then $f\circ g$ is an $BA$-map.

\begin{prop}\label{prop:lack_of_rigidity_persists_in_extensions}
Suppose that  $L\colon N/\Gamma\to N/\Gamma$ is a Anosov automorphism with simple spectrum that is not periodic data rigid. Then, if $N'$ is a nilpotent group containing a lattice $\Gamma'$ such that $N'/N'_{k}=N$ and $\Gamma'/(N'_{k}\cap \Gamma')=\Gamma$, and $L'\colon N'/\Gamma'\to N'/\Gamma'$ is an Anosov automorphism with simple spectrum inducing the map $L$ on the quotient $N/\Gamma$, then $L'$ is not periodic data rigid.
\end{prop}

\begin{proof}

By induction on the degree of nilpotency, it suffices to show the result when $N'$ is a central extension of $N$ so that $N'/\Gamma'\to N/\Gamma$ is a principal $\mathbb{T}^n$-bundle, where $n=\dim Z(N')$.

\begin{lem}
Suppose that $L\colon N/\Gamma\to N/\Gamma$ is an Anosov automorphism and $L'$ is a central extension of $L$ to the torus bundle $N'/\Gamma'$ that is an $A$-map, where $A$ is the restriction of $L'$ to the fiber through $e\Gamma'$.
Then there exists a $C^0$ neighborhood $\mc{U}$ of $L$ in $\Diff^{\infty}(N/\Gamma)$ such that if $f\in \mc{U}$ then there exists a smooth $A$-map $F$ covering $f$ such that $d_{C^{k}}(F,L')=O(d_{C^{k+1}}(f,L))$.
\end{lem}

In principle, the above lemma is immediate from Theorem 6.2 in \autocite{gogolev2015new}. We will not recapitulate that theorem in full here as it requires developing the language of classifying spaces. Vaguely, the theorem says that the only obstruction to the existence of such a map is at the level of homotopy. As $L'$ is an $A$-map covering $L$ and $f$ is homotopic to $L$, there is no obstruction to finding an $A$-map covering $f$. However, the theorem there does not assure us that $F$ is near to $L'$ nor that $F$ is as smooth as $L'$.

Consequently, we will give a different more explicit proof which produces $F$ that is $C^\infty$ near to $L'$.

\begin{proof}

Choose $\mc{U}$ sufficiently $C^0$-small so that if $f\in \mc{U}$ then for any $x\in N/\Gamma$ there exists a unique length minimizing geodesic $\gamma_x$ starting at $f(x)$ and ending at $L(x)$.  We write $E_x$ for the fiber of $N'/\Gamma'$ over $N/\Gamma$.

As $N'/\Gamma'$ is a principle $\mathbb{T}^n$-bundle over $N/\Gamma$, we may choose a $\mathbb{T}^n$-connection on this bundle, which gives an associated parallel transport. See, for instance, \autocite[Ch. 2 Sec. 3]{kobayashi1963foundations}. Consequently, parallel transport along $\gamma_x$ gives a map $P_x\colon E_{f(x)}\to E_{L(x)}$, which in a trivialization, is a translation on $\mathbb{T}^n$. Observe that these maps $P_x$ piece together to form a $C^{\infty}$ map $\Psi\colon  N'/\Gamma'\to N'/\Gamma'$, which is defined fiberwise and is an $\Id$-map of $\mathbb{T}^n$-bundles covering $L\circ f^{-1}$. Note that $\|\Psi\|_{C^{k}}=O(\|L\circ f^{-1}\|_{C^{k+1}})$ due to the definition of $\Psi$ via parallel transport. Consequently, $F=\Psi\circ L'$ is an $A$-map covering $f$ such that $d_{C^k}(F,L')=O(d_{C^{k+1}}(f,L))$. 
\end{proof}

\begin{lem}
Suppose that $F$ is an $A$-map of the bundle $N'/\Gamma'$ covering a map $f\colon N/\Gamma\to N/\Gamma$. Suppose $f$ has the same periodic data as $L$, then $F$ has the same periodic data as $L'$. Further, if $f$ is volume-preserving then so is $F$ and, in this case, $F$ has the same volume Lyapunov spectrum as $L'$.
\end{lem}

\begin{proof}
To begin, we show that $F$ and $L'$ have the same periodic data. Note that every periodic point of $F$ lies above a periodic point of $f$, i.e.\@ if $p$ is $F$ periodic then $\pi(p)$ is $f$ periodic. Suppose, for the moment, that $x$ is a fixed point. Then consider the differential of $F$ in a trivialization $U\times \mathbb{T}^n$. As $F$ is an $A$-map, the trivialization looks like:
\[
(u,z)\mapsto (f(u),Az+\phi(u)),
\]
for some $C^{\infty}$ function $\phi\colon U\to \mathbb{T}^n$, which gives the translational part of the map. Observe that the differential of the map is 
\begin{equation}\label{eqn:return_differential_of_F}
\begin{bmatrix}
Df & 0\\
D\phi & A
\end{bmatrix}.
\end{equation}
Similar considerations apply at all periodic points.
Consequently, the periodic data of $F$ is the union of the periodic data of $f$ with the periodic data of $A$. By Proposition \ref{prop:periodic_approximation}, this implies that all of the Lyapunov exponents for every invariant measure of $F$ coincide with those of $L'$. 

It remains to show that $F$ preserves a volume. Let $\omega$ be the volume form on $N/\Gamma$ preserved by $f$. There is a well-defined $n$-form, $\eta$, on $N'/\Gamma'$ coming from the principal torus bundle structure. Consider the form $(\pi^*\omega)\wedge \eta$, i.e.\@ the pullback of volume on the base wedged with the volume on the fiber. That this form is preserved is immediate due to the block form of the differential of $F$ obtained in equation (\ref{eqn:return_differential_of_F}).
\end{proof}

The following lemma shows that a conjugacy between $L'$ and the perturbation $F$ fibers over a map on $N/\Gamma$.

\begin{lem}
As above, suppose that $F\colon N/\Gamma\to N/\Gamma$ is an $A$-map and $L$ is an Anosov automorphism with simple spectrum that is also an $A$-map. If $H$ is a $C^1$ conjugacy between $F$ and $L$ then $H$ perserves the fibers of the fibration $N/\Gamma\to \underline{N}/\underline{\Gamma}$.
\end{lem}

\begin{proof}

By supposition as $H$ is $C^1$, $F$ has a number of dynamical features that it inherits from $L$. In particular, the stable and unstable subspaces are defined and have a continuous splitting into continuous one-dimensional subbundles. Let $E_i$ be a subspace associated to an eigenvalue $\lambda_i$ of $A$. We claim that $H$ intertwines the corresponding distributions $E_i$ associated to eigenvalues of the map $A$ acting on the fiber.
By inspecting equation (\ref{eqn:return_differential_of_F}), we see that for a periodic point of $F$ that the subspace $E_i$ is tangent to the fiber of the projection $N/\Gamma\to \underline{N}/\underline{\Gamma}$. This holds for each exponent arising from $A$ as the distribution defined by the splitting into one-dimensional subbundles is continuous and periodic points are dense. Similarly, the one-dimensional stable and unstable subspaces of $L$ arising from the action of $L$ on the center of $N$ are tangent to the fiber. Since we assumed simple spectrum, we see that $DH$ carries $E_i^{*,F}$ to $E_i^{*,L}$ for $*\in \{s,u\}$, and consequently $H$ preserves the fibers of the fibration.
\end{proof}

We now show that the perturbation $F$ of $L'$ that we constructed cannot be $C^1$ conjugate to  $L'$.
Suppose, for the sake of contradiction, that it is so that there exists a $C^1$ conjugacy $H$ satisfying $F\circ H=H\circ L'$. By the lemma, all three of these maps preserve the structure of the fibration $N'/\Gamma'\to N/\Gamma$. So, each descends to a map $N/\Gamma\to N/\Gamma$ and these quotient maps satisfy $\underline{F}\circ \underline{H}=\underline{H}\circ \underline{L'}$. By assumption, we already know what two of these maps are, so we have that $f\circ \underline{H}=\underline{H}\circ L$. 
But, as $H$ was assumed to be a $C^1$ map, and we showed that it perserves the fibers of the bundle $N'/\Gamma'\to N/\Gamma$, we obtain that $\underline{H}$ is $C^1$. However, as $h$ is not $C^1$, this contradicts Proposition \ref{prop:same_regularity}, which shows that if one conjugacy between $f$ and $L$ is $C^1$ then all conjugacies are $C^1$.
\end{proof}

Using the above we can prove the following theorem.

\begin{thm}\label{thm:non_rigid}
Suppose that $L\colon N/\Gamma\to N/\Gamma$ is an Anosov automorphism with simple spectrum such that either $L$ is not irreducible or the exponents of $L$ are not sorted. Then $L$ is not locally Lyapunov spectrum rigid or periodic data rigid.
\end{thm}

\begin{proof}
There exists a term $N_k$ in the lower central series such that the map induced by $L$ on $(N/N_k)/\pi(\Gamma)$ satisfies the hypotheses of either Proposition \ref{prop:out_of_order_reducible_prop} or Proposition \ref{prop:irreducibility_fails}. Consequently the induced map on the quotient is not periodic data rigid. By Proposition \ref{prop:lack_of_rigidity_persists_in_extensions}, we conclude that $L$ is not periodic data rigid or Lyapunov spectrum rigid either.
\end{proof}

\section{Examples, counterexamples, and curiosities}\label{sec:examples}

In this section we construct several examples of Anosov automorphisms of nilmanifolds illustrating novelties of the nilmanifold case. Most importantly, we construct an automorphism of a two-step nilmanifold that is irreducible and has sorted spectrum, which shows that the rigidity theorem in this paper is not vacuous. We also describe nilmanifolds that do not admit locally rigid automorphisms. We also give an example of an automorphism where the conjugacy is $C^{1+}$ along the unstable foliation but not the stable foliation.

\subsection{A locally rigid automorphism}

In his seminal survey of smooth dynamics, \autocite{smale1967differentiable}, Smale gave two examples of Anosov automorphisms on a particular two-step nilmanifold. However, Smale's examples are not rigid because on the quotient torus they are the direct sum of two automorphisms on $\mathbb{T}^2$. Consequently, they are reducible.
Though Smale's particular examples are not rigid, it is possible to construct a rigid example on that nilmanifold. 

\begin{example}[Smale's nilmanifold]
Consider the Lie group given by the product of two copies of the Heisenberg group, $H_1$ and $H_2$. Write $X_i,Y_i,Z_i$, where $[X_i,Y_i]=Z_i$ for the usual basis of the Lie algebra $\mf{h}_i$ of $H_i$.  We then construct an automorphism of the Lie group and algebra by setting:
\begin{align*}
X_1&\mapsto (26+15\sqrt{3})X_1+(8733+5042\sqrt{3})Y_1\\
Y_1&\mapsto (71+41\sqrt{3})X_1+(28901 + 16686\sqrt{3})Y_1\\
Z_1&\mapsto (262087+151316\sqrt{3})Z_1\\
X_2&\mapsto (26-15\sqrt{3})X_2+(8733-5042\sqrt{3})Y_2\\
Y_2&\mapsto (71-41\sqrt{3})X_2+(28901 - 16686\sqrt{3})Y_2\\
Z_2&\mapsto (262087-151316\sqrt{3})Z_2.
\end{align*}
Note that for $X_2,Y_2,Z_2$ we have just changed $+\sqrt{3}$ to $-\sqrt{3}$. 
This defines an automorphism of $H_1\times H_2$. As in Smale's case, we define a lattice by first defining a lattice in $\mf{h}_1\oplus \mf{h}_2$. View this Lie algebra as matrices of the form
\[
\begin{bmatrix}
0 & X & Z\\
0 & 0 & Y\\
0 & 0 & 0
\end{bmatrix}\oplus
\begin{bmatrix}
0 & X & Z\\
0 & 0 & Y\\
0 & 0 & 0
\end{bmatrix}.
\]
Now for the lattice, take the elements where the first matrix has entries in $\mc{O}(\Q(\sqrt{3}))$ and for the second factor take the conjugate entries to those chosen in the first factor. Explicitly, this lattice consists of matrices with entries in $\Z+\Z\sqrt{3}$ of the form 
\[
\begin{bmatrix}
0 & a+ b\sqrt{3} & e+f\sqrt{3}\\
0 & 0 & c+d\sqrt{3}\\
0 & 0 & 0\\
\end{bmatrix}\oplus
\begin{bmatrix}
0 & a-b\sqrt{3}& e-f\sqrt{3}\\
0 & 0 & c-d\sqrt{3}\\
0 & 0 & 0\\
\end{bmatrix}.
\]
The automorphism specified preserves this lattice in $\mf{h}$ because the automorphism is defined over $\Z[\sqrt{3}]$, and is invertible because its determinant is in $\Z[\sqrt{3}]^{\times}$. Thus it can be shown that the lattice, $\Gamma$, generated by the exponential image of this lattice in $H_1\oplus H_2$ is preserved by the corresponding automorphism of $H_1\oplus H_2$. Thus we obtain an automorphism of $(H_1\oplus H_2)/\Gamma$.

We now check that this automorphism is sorted and totally irreducible. First we need irreducibility over $\Q$ of the action in the base and in the fiber. The map in the fiber is a map of a $2$-torus and hence is irreducible as it is hyperbolic. The map in the fiber has eigenvalues $262087+151316\sqrt{3}$
and $262087-151316\sqrt{3}$.

In order to compute the map in the base, it suffices to study the automorphism we obtain of the lattice when we have quotiented out by the fiber. Represent an element of the basis for the quotient lattice by $(a,b,c,d)$ corresponding to the matrix:
\[
\begin{bmatrix}
0 & a+ b\sqrt{3} & *\\
0 & 0 & c+d\sqrt{3}\\
0 & 0 & 0\\
\end{bmatrix}\oplus
\begin{bmatrix}
0 & a-b\sqrt{3}& *\\
0 & 0 & c-d\sqrt{3}\\
0 & 0 & 0\\
\end{bmatrix}.
\]
Written with respect to the coordinates $(a,b,c,d)$, the automorphism is:
\begin{equation}
\begin{bmatrix}
26 & 45& 71 & 123 \\
15 & 26 & 41 & 71 \\
8733 & 15126 & 28901 & 50058\\
5042 & 8733 & 16686 & 28901
\end{bmatrix}.
\end{equation}
We then calculate the eigenvalues of this matrix using Mathematica. They are approximately:
\begin{center}
\texttt{57844.9, 9.06171, 0.0193643, 0.0000985198}
\end{center}
Whereas the eigenvalues in the fiber are approximately:
\begin{center}
\texttt{524174, .00000190779(=1.9*10e-6)}
\end{center}
Thus we see that the stable and unstable spectrum are sorted.
Finally, one may compute the characteristic polynomial of this matrix and check that it is irreducible over $\Q$. This automorphism is sorted and totally irreducible and so is locally Lyapunov spectrum rigid by the main theorem.
\end{example}

\subsection{Non-rigid families}

Using the necessity of sorted spectrum for local rigidity, we can exhibit nilmanifolds such that no Anosov automorphism with simple spectrum on that nilmanifold is Lyapunov spectrum rigid. For this construction we use a free nilpotent Lie group.
Much additional information about automorphisms of free nilpotent groups is contained in \autocite{payne2009anosov}. 
 We consider the two-step free nilpotent group on $3$ generators, $N_{3,2}$. Explicitly, 
define a Lie algebra $\mf{n}$ via three generators $x,y,z$. These generators bracket to linearly independent elements $[x,y]$, $[x,z]$ and $[y,z]$, and we stipulate that $\mf{n}_{3}=0$, so that this Lie algebra is two-step. In fact, we see that $\{x,y,z,[x,y],[x,z],[y,z]\}$ is a basis for $\mf{n}$ with rational structure coefficients, and hence its $\Z$-span is a lattice in $\mf{n}$. The exponential image of this lattice generates a lattice in $N$, which we call $\Lambda$. 
As $N_{3,2}$ is free, an automorphism $L$ of $N_{3,2}$ is determined by the induced map $\underline{L}$ on $N_{3,2}/Z(N_{3,2})$, which we identify with $\R^3$. If $\underline{L}$ has eigenvalues $0<\abs{\lambda_1}<1<\abs{\lambda_2}<\abs{\lambda_3}$, and preserves volume so that $\abs{\lambda_1}\abs{\lambda_2}\abs{\lambda_3}=1$, then we see that $L\mid_{Z(N_{2,3})}$ has as its exponents $\abs{\lambda_i}^{-1}$, for $i\in \{1,2,3\}$. In particular, $1>\abs{\lambda_2}^{-1}>\abs{\lambda_1}$, and so such an automorphism has its stable exponents out of order. Consequently, no such automorphism of $N_{3,2}/\Lambda$ is Lyapunov spectrum rigid.

Despite such an automorphism not being rigid, we can still construct an automorphism $L$ on this nilmanifold so that a conjugacy between $L$ and a sufficiently small perturbation with the same spectrum is $C^{1+}$ along the unstable foliation even though it is not $C^{1+}$ along the stable foliation. 

\begin{example}[Free nilpotent group]

Consider the automorphism of $\mf{n}$ defined by:

\begin{align*}
x&\mapsto y\\
y&\mapsto z\\
z&\mapsto -x+8y-z\\
[x,y]&\mapsto [y,z]\\
[x,z]&\mapsto [x,y]-[y,z]\\
[y,z]&\mapsto [x,z]-8[y,z].\\
\end{align*}
This automorphism preserves the lattice in $\mf{n}$ and hence extends to an automorphism of the group $N_{3,2}$ preserving the lattice, $\Lambda$, generated by the exponential image this lattice. The differential at the identity of this automorphism is:
\[
\begin{bmatrix}
0 & 0 & -1 & 0 & 0 & 0\\
1 & 0 & 8 & 0 & 0 & 0\\
0 & 1 & -1 & 0 & 0 & 0\\
0 & 0 & 0 & 0 & 1 & 0\\
0 & 0 & 0 & 0 & 0 & 1\\
0 & 0 & 0 & 1 & -1 & -8
\end{bmatrix}
\]
The map on the $\mf{n}/[\mf{n},\mf{n}]$ is given by the upper left hand $3\times 3$ submatrix. This map has characteristic polynomial $x^3+x^2-8x+1$, which defines a totally real extension of degree $3$ over $\R$. In particular this polynomial is irreducible over $\Q$. Similarly the bottom right corner has characteristic polynomial $1+x+8x^2-x^3$, which is also irreducible over $\Q$. One may compute the eigenvalues of the matrix numerically to find that they are, in order of magnitude:

\begin{center}
\texttt{-7.85652, -3.4227, 2.29542, -0.435651, 0.292167, 0.127283}.
\end{center}
The eigenvalues coming from the three by three matrix in the upper left are the second, third, and last. The induced map on the quotient and fiber are both irreducible and the Lyapunov exponents on the unstable manifold are sorted. The exponents on the stable manifold are not sorted as the strongest contraction happens in the base.
\end{example}

\subsection{Rigidity of non-linear examples}

It is natural to ask whether there are non-linear Anosov diffeomorphisms that exhibit volume Lyapunov spectrum rigidity. Recent examples constructed by Erchenko show that even in the case of $\mathbb{T}^2$ that two diffeomorphisms may have the same volume Lyapunov spectrum and not be $C^1$ conjugate. This follows from \autocite[Thm. 1.1]{erchenko2019flexibility}, which gives examples of diffeomorphisms with the same volume Lyapunov spectrum but different Lyapunov spectrum for the measure of maximal entropy. See Question 2 of \autocite{erchenko2019flexibility} for a more refined question relating rigidity to an associated pressure function.

\printbibliography

\appendix

\end{document}